\documentclass[a4paper,10pt]{article}
\usepackage[english]{babel}
\usepackage[latin1]{inputenc}
\usepackage[T1]{fontenc}
\usepackage{amsmath}
\usepackage{comment}
\usepackage{amsfonts}
\usepackage{amsthm}
\usepackage{mathrsfs}
\usepackage{amssymb}
\usepackage{graphicx}
\usepackage{tikz-cd}
\usetikzlibrary{calc}
\usepackage{marginnote}

\theoremstyle{definition}
\newtheorem{Def}{Definition}[section]
\newtheorem{Ex}[Def]{Example}
\newtheorem{Rmk}[Def]{Remark}

\theoremstyle{plain}
\newtheorem{Prop}[Def]{Proposition}
\newtheorem{Thm}[Def]{Theorem}
\newtheorem{Lemma}[Def]{Lemma}
\newtheorem{Cor}[Def]{Corollary}

\newcommand{\R}{\mathbb{R}}

\newcommand{\C}{\mathbb{C}}
\newcommand{\Z}{\mathbb{Z}}
\newcommand{\N}{\mathbb{N}}
\newcommand{\B}{\mathcal{B}}
\newcommand{\Bp}{\mathcal{B}^*}

\newcommand{\lb}{\underline{b}}
\newcommand{\ub}{\overline{b}}
\newcommand{\subsetc}{\hookrightarrow}
\newcommand{\J}{\mathcal{J}}
\newcommand{\Sc}{\mathcal{S}}
\newcommand{\F}{\mathcal{F}}

\newcommand{\snA}{\mathscr{B}_n}
\newcommand{\ban}{\mathscr{B}}
\newcommand{\nA}{\mathbf{A}}
\newcommand{\nB}{\mathbf{B}}
\newcommand{\clas}{\mathscr{C}}
\newcommand{\hypu}{(\mathcal{H}_1)}
\newcommand{\hypd}{(\mathcal{H}_2)}
\newcommand{\hypt}{(\mathcal{H}_3)}
\renewcommand{\epsilon}{\varepsilon}

\title{A Functorial Approach to Multi-Space Interpolation with Function Parameters}
\author{T. Lamby\footnote{Universit\'e de Li\`ege, D\'epartement de math\'ematique -- zone Polytech 1, 12 all\'ee de la D\'ecouverte, B\^at. B37, B-4000 Li\`ege. thomas.lamby@uliege.be} and Samuel Nicolay\footnote{Universit\'e de Li\`ege, D\'epartement de math\'ematique -- zone Polytech 1, 12 all\'ee de la D\'ecouverte, B\^at. B37, B-4000 Li\`ege. Orcid ID: 0000-0003-0549-0566. S.Nicolay@uliege.be}}

\begin{document}

\maketitle

\begin{abstract}
We introduce an extension of interpolation theory to more than two spaces by employing a functional parameter, while retaining a fully functorial and systematic framework. This approach allows for the construction of generalized intermediate spaces and ensures stability under natural operations such as powers and convex combinations. As a significant application, we demonstrate that the interpolation of multiple generalized Sobolev spaces yields a generalized Besov space. Our framework provides explicit tools for handling multi-parameter interpolation, highlighting both its theoretical robustness and practical relevance.
\end{abstract}


\section{Introduction}
The theory of interpolation of vector spaces originated from an observation made by Marcinkiewicz \cite{Marcinkiewicz:39}, which led to the formulation of the Riesz-Thorin theorem \cite{Riesz:27,Thorin:48}. Roughly speaking, if a linear operator is continuous on both $L^p$ and $L^q$ spaces, then it is also continuous on $L^r$ for any intermediate value $r$ between $p$ and $q$. Real interpolation techniques \cite{Bergh:76,Bru:91} provide a formal framework for this idea, often via the $K$-method. Consider two Banach spaces $A_0$ and $A_1$, continuously embedded in a Hausdorff topological vector space, so that $A_0\cap A_1$ and $A_0+A_1$ are well-defined Banach spaces. The $K$-functional is defined by
\[
 K(t,a)= \inf_{a=a_0+a_1} \{\|a_0\|_{A_0} + t \|a_1\|_{A_1}\}, 
\]
for $t>0$ and $a\in A_0+A_1$. For $0<\alpha<1$ and $q\in [1,\infty]$, an element $a$ belongs to the interpolation space $[A_0,A_1]_{\alpha,q}$ if $a\in A_0+A_1$ and
\begin{equation}\label{eq:K method orig}
 (2^{-j\alpha}K(2^j,a))_{j\in \Z} \in \ell^q.
\end{equation}

Real interpolation methods have been extended by introducing a function parameter \cite{Peetre:68,Gustavsson:78,Brudnyi:81,Cwikel:81,Janson:81,Ovchinnikov:84,Merucci:84,Maligranda:86,Persson:86,Lamby:24}, replacing the sequence $(2^{-j\alpha})_{j\in \Z}$ in \eqref{eq:K method orig} with a Boyd function. Related formulations appear in \cite{Kreit:13,Loosveldt:18}, and connections among these approaches have been investigated in \cite{Lamby:22}.

An additional extension arises when interpolating more than two spaces. Early work on this topic can be found in \cite{Foi:61}, with further developments in \cite{Ker:66,Yos:70,Spa:74,Ase:01,Triebel:78}. In this paper, we consider an extension of the classical interpolation framework to multiple spaces with a function parameter. While a general equivalence between the $K$- and $J$-methods has been proposed in \cite{Ase:97}, we adopt a functorial perspective that allows explicit constructions in the multi-space context. A key aspect of this approach is that Boyd functions permit the inclusion of limiting cases, generalizing the classical setting when $n=1$.

We illustrate the theory by computing the generalized interpolation of two Lorentz spaces and by showing that the interpolation of multiple generalized Sobolev spaces produces a generalized Besov space. Moreover, we discuss how this framework can be applied to reiteration formulas, in the spirit of \cite{Ase:01}, and how it affects the preservation of smooth function spaces under multi-space interpolation \cite{Peetre:68, Persson:86, Ase:01}.

The structure of the paper is as follows. In Section~\ref{sec:Boyd}, we introduce Boyd functions and the basic notions of interpolation functors. Section~\ref{sec:interpol} presents the interpolation of multiple spaces with function parameters. Section~\ref{sec:equiv} examines the equivalence between the $K$- and $J$-methods in this context, and Section~\ref{sec:powth} establishes the power theorem and discusses stability properties of the associated spaces. Finally, Sections~\ref{sec:sobolev} and \ref{sec:lorentz} illustrate applications to generalized Sobolev and Lorentz spaces and to multi-space reiteration results.

\section{Boyd functions}\label{sec:Boyd}
We hereby review fundamental characteristics of Boyd functions. For a more in-depth exploration, the reader is referred to \cite{Lamby:22,Merucci:84} and the references therein. We denote by $L_*^q(E)$ the space of measurable functions $f$ defined on $E$, satisfying the condition that $(\int_E |f|^q  dt/t)^{1/q}$ is finite; with the conventional adjustment made for the case when $q=\infty$.

\begin{Def}\label{def:Boyd}
A function $\phi:(0,\infty)\to (0,\infty)$ is classified as a Boyd function if it satisfies the following conditions:
\begin{itemize}
\item $\phi(1)=1$,
\item $\phi$ is continuous,
\item for any $t>0$,
\[
 \bar\phi(t) = \sup_{s>0} \frac{\phi(ts)}{\phi(s)}<\infty.
\]
\end{itemize}
The set of Boyd functions is denoted by $\B$.
\end{Def}
The function $\bar \phi$ is sub-multiplicative, Lebesgue-measurable and it holds that $1/\bar\phi(1/t)\le \phi(t)\le \bar\phi(t)$ and $1/\phi(t)\le 1/\bar\phi(1/t)$.
\begin{Def}\label{def:indices}
For a given $\phi\in \B$, the lower and upper Boyd indices \cite{Matu:60,Boyd:69} are defined as follows:
\[
 \lb(\phi) = \sup_{t\in (0,1)} \frac{\log \bar\phi(t)}{\log t}
 = \lim_{t\to 0^+} \frac{\log \bar\phi(t)}{\log t}
\]
and
\[
 \ub(\phi) = \inf_{t\in (1,\infty)} \frac{\log \bar\phi(t)}{\log t}
 = \lim_{t\to \infty} \frac{\log \bar\phi(t)}{\log t},
\]
respectively.
\end{Def}
For instance, considering $\gamma\ge 0$ and $\theta\in \R$, if we set
\[
\phi(t)=t^\theta (1+|\log t|)^\gamma,
\]
for $t>0$, it can be readily verified that $\lb(\phi)=\ub(\phi)=\theta$. If $\phi\in\B$, for $\epsilon>0$ and $R>0$, there exists a constant $C>0$ such that
\[
 C^{-1} r^{\ub(\phi)+\epsilon} \le \phi(r)\le C r^{\lb(\phi)-\epsilon},
\] 
for any $r\in(0,R)$ and a constant $C>0$ such that
\[
 C^{-1} r^{\lb(\phi)-\epsilon} \le \phi(r)\le C r^{\ub(\phi)+\epsilon},
\] 
for any $r\ge R$. Furthermore, for such a function $\phi$, the following properties hold:
\begin{itemize}
\item $\lb(\phi)>0$ $\Leftrightarrow$ $\bar\phi\in L_*^1(0,1)$ $\Leftrightarrow$ $\displaystyle \lim_{t\to 0^+}\bar\phi(t)=0$,
\item $\ub(\phi)<0$ $\Leftrightarrow$ $\bar\phi\in L_*^1(1,\infty)$ $\Leftrightarrow$ $\displaystyle \lim _{t\to \infty}\bar\phi(t)=0$,
\item $\lb(\phi)>0$ $\Rightarrow$ $\phi\in L^\infty(0,1)$,
\item $\ub(\phi)<0$ $\Rightarrow$ $\phi\in L^\infty(1,\infty)$.
\end{itemize}

\begin{Def}
A function $\phi:(0,\infty)\to (0,\infty)$ is considered Boyd-regular if it satisfies the following conditions:
\begin{itemize}
\item $\phi(1)=1$,
\item $\phi\in C^1(0,\infty)$,
\item the inequalities
\[
 0< \inf_{t>0} t \frac{|D\phi(t)|}{\phi(t)}
 \le \sup_{t>0} t \frac{|D\phi(t)|}{\phi(t)} <\infty
\]
hold.
\end{itemize}
The set of Boyd-regular functions is denoted by $\Bp$.
\end{Def}
It is known that $\Bp\subset \B$. The subset of functions $\phi\in \Bp$ that are strictly increasing (resp.\ strictly decreasing) is denoted by $\Bp_+$ (resp.\ $\Bp_-$).

Given two functions $f$ and $g$ defined on $(0,\infty)$, the notation $f\asymp g$ is used when there exists a constant $C>0$ such that $C^{-1} g(t)\le f(t)\le C g(t)$ for any $t>0$. If $\phi\in \B$ is such that $\lb(\phi)>0$ (resp.\ $\ub(\phi)<0$), then there exists $\xi\in \Bp_+$ (resp.\ $\xi\in \Bp_-$) such that $\xi^{-1}\in \Bp_+$ (resp.\ $\xi^{-1}\in \Bp_-$) and $\phi\asymp \xi$ (see \cite{Merucci:84}).

\section{Interpolation}\label{sec:interpol}
In the theory of real interpolation, the general correspondence from two to several Banach spaces is not always feasible. However, in a multitude of specific instances, comparable outcomes can be ascertained.
\begin{Def}
We denote $\nA$ (resp.\ $\nB$) as a ($n+1$)-tuples of Banach spaces $A_0,\dotsc,A_n$ (resp.\ $B_0,\dotsc,B_n$) such that each $A_j$(resp.\ each $B_j$) can be linearly and continuously imbedded in a Hausdorff topological vector space $\mathcal{A}$ (resp.\ $\mathcal{B}$): $\nA=(A_0,\dotsc,A_n)$. Given two such sets $\nA$ and $\nB$, morphisms $T:\nA\rightarrow \nB$ are homomorphisms $T:\mathcal{A}\rightarrow \mathcal{B}$ such that $T:A_j\rightarrow B_j$ are continuous homomorphisms.
\end{Def}
\begin{Def}
The standard norm on $\Sigma(\nA)=A_0+\dotsb+A_n$ is defined by
\[
 a\mapsto \|a\|_{\Sigma(\nA)}= \inf \{\|a_0\|_{A_0}+\dotsb+\|a_n\|_{A_n}\},
\]
where the infimum is taken over all decompositions $a=a_0+\dotsb+a_n$, with $a_j\in A_j$.
\end{Def}
\begin{Def}
The standard norm on $\Delta(\nA)=A_0\cap\dotsb\cap A_n$ is defined by
\[
 a\mapsto \|a\|_{\Delta(\nA)}= \max\{ \|a\|_{A_0},\dotsc,\|a\|_{A_n}\}.
\]
\end{Def}
Let $\snA$ be the collection of the elements $\nA$; if $\ban$ denotes the category of Banach spaces, it is evident that $\Sigma$ and $\Delta$ are functors from $\snA$ to $\ban$.

\begin{Def}
An interpolation functor of order $n$ is a functor $F:\snA\rightarrow \ban$ such that
\begin{itemize}
\item[(i)]$\Delta(\nA)\subsetc F(\nA)\subsetc \Sigma(\nA)$ for any $\nA$ in $\snA$,
\item[(ii)]$F(T)=T\vert_{F(\nA)}$ for any morphism $T:\nA\rightarrow \nB$.% in $\snA$.
\end{itemize}
\end{Def}

\begin{Def}
Let $f$ be a function from $(0,\infty)^{n+1}$ to $(0,\infty)$; %$\overline{\R}^{n+1}_+\rightarrow \overline{\R}_+$.
an interpolation functor $F$ is of type $f$ if there exists a constant $C\ge 1$ such that
\[
 \|T\|_{F(\nA),F(\nB)}\le C f\left(\|T\|_{A_0,B_0},\dotsc,\|T\|_{A_n,B_n}\right),
\]
for any morphism $T:\nA\to \nB$. % in $\snA$.
If, moreover,
\begin{itemize}
\item $f(t_0,\dotsc,t_n)=\max\{t_0,\dotsc,t_n\}$, we say that $F$ is bounded,
\item $C=1$, we say that $F$ is exact.
\end{itemize}
\end{Def}
\begin{Def}
Given $\phi_1,\dotsc,\phi_n\in \B$, $p\in[1,\infty]$, we will say that condition $\hypu$ is satisfied when 
\[
 \Big(\int_{(0,\infty)^n}\big(\frac{1}{\phi_1(t_1)\dotsb\phi_n(t_n)}\min\{1,t_1,\dotsc,t_n\}\big)^p \, \frac{dt_1}{t_1}\dotsb \frac{dt_n}{t_n}\Big)^{1/p}<\infty,
\]
with the usual modification in the case $p=\infty$.
\end{Def}
\begin{Def}
Given $\phi_1,\dotsc,\phi_n\in \B$, we say that condition $\hypd$ is satisfied if 
\[
 \min\{\lb(\phi_1),\dotsc,\lb(\phi_n)\}>0
 \quad\text{ and }\quad
 \max\{\ub(\phi_1),\dotsc , \ub(\phi_n)\}<1
 \]
 and that condition $\overline\hypd$ is satisfied if
 \[
  \min\{\lb(\phi_1),\dotsc,\lb(\phi_n)\}\ge 0
  \quad\text{ and }\quad
  \max\{\ub(\phi_1),\dotsc , \ub(\phi_n)\}\le 1.
\]
\end{Def}
\begin{Rmk}
Using the properties of Boyd functions, one can readily show that $\hypd$ implies $\hypu$ and if $p=\infty$, $\overline\hypd$ implies $\hypu$.
\end{Rmk}

Let $\sim$ be the equivalence relation defined by $(t_0,\dotsc,t_n)\sim(s_0, \dotsc,s_n)$ if and only if there exists $\lambda>0$ such that $t_j=\lambda s_j$ for all $j\in\{0,\dotsc,n\}$.

\begin{Def}
Let  $p\in [1,\infty]$, $\phi_1,\dotsc,\phi_n\in \B$ satisfying $\overline \hypd$ %be such that $0\le \lb(\phi_1)+\dotsb+\lb(\phi_n)$ and $\ub(\phi_1)+\dotsb+\ub(\phi_n)\le 1$
and $f$ be a function from $(0,\infty)^{n+1}$ to $(0,\infty)$; one defines
\[
 \Phi_p^{\phi_1,\dotsc,\phi_n}(f)
 =\Big(\int_{\mathbb{P}^{n}_+}\big(\frac{\phi_1(t_0)\dotsb\phi_n(t_0)}{t_0\phi_1(t_1)\dotsb\phi_n(t_n)}f(t_0,t_1,\dotsc,t_n)\big)^p \, d\omega(t)\Big)^{1/p},
\]
with the usual modification in the case $p=\infty$, where $\omega$ is the Haar measure on $\mathbb{P}^{n}_+=(0,\infty)^{n+1}/\sim$.
\end{Def}
\begin{Rmk}\label{rem:BanL}
For the Haar measure $\omega$ in the previous definition, we have
\[
 d\omega(t)=\sum_{j=0}^n(-1)^j \, \frac{dt_0}{t_0}\wedge\dotsb \wedge\check{\frac{dt_j}{t_j}}\wedge\dotsb\wedge\frac{dt_n}{t_n},
\]
where the symbol $\,\check{}\,$ indicates a missing factor. Consequently,
\begin{equation}\label{eq:Phip}
 \Phi_p^{\phi_1,\dotsc,\phi_n}(f)=
 \Big(\int_{(0,\infty)^n}\big(\frac{f(1,t_1,\dotsc,t_n)}{\phi_1(t_1)\dotsb \phi_n(t_n)}\big)^p \, \frac{dt_1}{t_1}\dotsb \frac{dt_n}{t_n}\Big)^{1/p},
% \Big(\int_{(0,\infty)^n}\big(\frac{1}{\phi_1(t_1)}\dotsb \frac{1}{\phi_n(t_n)}f(1,t_1,\dotsc,t_n)\big)^p \, \frac{dt_1}{t_1}\dotsb \frac{dt_n}{t_n}\Big)^{1/p},
\end{equation}
with the usual modification in the case $p=\infty$.
\end{Rmk}

\begin{Def}
Given $a\in\nA$ and $t\in (0,\infty)^{n+1}$, let
 \[
 K(t,a)=\inf_a \{t_0\|a_0\|_{A_0}+t_1\|a_1\|_{A_1}+\dotsb +t_n\|a_n\|_{A_n}\},
\]
where the infimum is taken over all the decompositions $a=a_0+\dotsb+a_n$, $a_j\in A_j$. In the same way, for $a\in\Delta(\nA)$, we set
\[
 J(t,a)=\max\{ t_0\|a\|_{A_0},t_1\|a\|_{A_1},\dotsc,t_n\|a\|_{A_n}\}.
\]
\end{Def}
It follows that
\[
 K(t,a)\le \max\{\frac{t_0}{s_0},\dotsc,\frac{t_n}{s_n}\}K(s,a),
\]
\[
 J(t,a)\le \max\{\frac{t_0}{s_0},\dotsc,\frac{t_n}{s_n}\}J(s,a)
\]
and
\[
 K(t,a)\le \max\{\frac{t_0}{s_0},\dotsc,\frac{t_n}{s_n}\}J(s,a).
\]
 
\begin{Def}
Let $p\in [1,\infty]$ and $\phi_1,\dotsc,\phi_n\in \B$ satisfying $\overline\hypd$%
% be such that $0\le \lb(\phi_1)+\dotsb+\lb(\phi_n)$ and $\ub(\phi_1)+\dotsb+\ub(\phi_n)\le 1$
; we define $K_p^{\phi_1,\dotsc,\phi_n}(\nA)$ as the set of all $a\in \Sigma(\nA)$ for which
\[
 \|a\|_{K_p^{\phi_1,\dotsc,\phi_n}(\nA)}=\Phi_p^{\phi_1,\dotsc,\phi_n}\big(K(t,a)\big)<\infty.
\]
\end{Def}
\begin{Prop}\label{Prop Kfct}
 Let $p\in [1,\infty]$ and $\phi_1,\dotsc,\phi_n\in \B$ satisfying $\hypu$; the functor $K_p^{\phi_1,\dotsc,\phi_n}$ is an exact interpolation functor of type $f$, where
 \[
  f(t_0,\dotsc,t_n)=t_0\overline{\phi_1}(t_1/t_0)\dotsb \overline{\phi_n}(t_n/t_0).
\]

Moreover, 
\[
 K(1,t_1,\dotsc,t_n,a)\le C \phi_1(t_1)\dotsb \phi_n(t_n)\|a\|_{K_p^{\phi_1,\dotsc,\phi_n}(\nA)}.
\]
\end{Prop}
\begin{proof}
Using $\hypu$, we have
 \begin{align*}
 & \big(\int_{(0,\infty)^n}(\frac{1}{\phi_1(u_1s_1)\dotsb\phi_n(u_ns_n)}\min\{1,u_1,\dotsc,u_n\})^p \, \frac{du_1}{u_1}\dotsb \frac{du_n}{u_n}\big)^{1/p}  \\
 &\qquad K(1,s_1,\dotsc,s_n,a)\\
 &\le \Big(\int_{(0,\infty)^n} \big(\frac{1}{\phi_1(t_1)\dotsb \phi_n(t_n)}\min\{1,\frac{t_1}{s_1},\dotsc,\frac{t_n}{s_n} \} \\
&\qquad K(1,s_1,\dotsc,s_n)\big)^p \, \frac{dt_1}{t_1}\dotsb \frac{dt_n}{t_n}\Big)^{1/p}\\
 &\le \|a\|_{K_p^{\phi_1,\dotsc,\phi_n}(\nA)},
\end{align*}
so that
\begin{equation}\label{equation 1 dans thm K,phi,p foncteur}
 K(1,s_1,\dotsc,s_n,a)\le C \phi_1(s_1)\dotsb\phi_n(s_n)\|a\|_{K_p^{\phi_1,\dotsc,\phi_n}(\nA)}.
\end{equation}
By using (\ref{equation 1 dans thm K,phi,p foncteur}) with $s_j=1$, we get
\[
 \|a\|_{\Sigma(\nA)}= K(1,a)\le C \|a\|_{K_p^{\phi_1,\dotsc,\phi_n}(\nA)},
\]
so that $K_p^{\phi_1,\dotsc,\phi_n}(\nA)\subsetc \Sigma(\nA)$.

Now, since 
\[
 K(1,t_1,\dotsc,t_n,a)\le \min\{1,t_1,\dotsc,t_n\}\|a\|_{\Delta(\nA)},
\]
we have
\[
 \|a\|_{K_p^{\phi_1,\dotsc,\phi_n}(\nA)}\le C \|a\|_{\Delta(\nA)},
\]
and thus, using $\hypu$, $\Delta(\nA)\subsetc K_p^{\phi_1,\dotsc,\phi_n}(\nA)$.

Now, let $\nA$ and $\nB$ be two couples in $\snA$ and let $T:\nA \to \nB$; we can write
\begin{align*}
K^{\nB}(1,t_1,\dotsc,t_n,Ta)\le \|T\|_{A_0,B_0} K^{\nA}(1,t_1\frac{\|T\|_{A_1,B_1}}{\|T\|_{A_0,B_0}},\dotsc,t_n\frac{\|T\|_{A_n,B_n}}{\|T\|_{A_0,B_0}},a),
\end{align*}
so that, for $\alpha_j=\frac{\|T\|_{A_j,B_j}}{\|T\|_{A_0,B_0}}$, we have
\begin{align*}
&\|Ta\|_{K_p^{\phi_1,\dotsc,\phi_n}(\nB)} \\
%&\le \|T\|_{A_0,B_0}\left(\int_{(0,\infty)^n} \left(\frac{1}{\phi_1(t_1)\dotsb \phi_n(t_n)}K^{\nA}(1,t_1\alpha_1,\dotsc,t_n\alpha_n,a)\right)^p\frac{dt_1}{t_1}\dotsb \frac{dt_n}{t_n}\right)^{1/p}\\
%=
&\le \|T\|_{A_0,B_0}\Big(\int_{(0,\infty)^n} \big(\frac{1}{\phi_1(u_1/\alpha_1)\dotsb \phi_n(u_n/\alpha_n)}K^{\nA}(1,u_1,\dotsc,u_n,a)\big)^p \\
&\qquad \frac{du_1}{u_1}\dotsb\frac{du_n}{u_n}\Big)^{1/p}\\
&\le \|T\|_{A_0,B_0}\overline{\phi_1}(\alpha_1)\dotsb \overline{\phi_n}(\alpha_n)\|a\|_{K_p^{\phi_1,\dotsc,\phi_n}(\nA)},
\end{align*}
thus completing the demonstration.
\end{proof}

\begin{Rmk}
If $\hypu$ is not satisfied, then $K_p^{\phi_1,\dotsc,\phi_n}(\nA)=\{0\}$ and thus $K_p^{\phi_1,\dotsc,\phi_n}$ is not an interpolation functor, since $\Delta(\nA)\subsetc K_p^{\phi_1,\dotsc,\phi_n}(\nA)$ is not true when $\Delta(\nA)\neq \{0\}$.
\end{Rmk}

\begin{Rmk}
The fact that we employ assumption $\hypu$ rather than a more general type of $\hypd$ in Proposition \ref{Prop Kfct} indicates that we are addressing ``limiting cases'', following \cite{Lamby:25}. This provides an additional justification for using Boyd functions to generalize real interpolation spaces, since the conditions guaranteeing, for instance, that the equivalence theorem remains valid (see \cite{Lamby:25} for the case $n=1$, and Remark \ref{rmk lim cases n>1}) can be stated straightforwardly in terms of Boyd functions.
\end{Rmk}

\begin{Def}
 Let $p\in [1,\infty]$ and $\phi_1,\dotsc,\phi_n\in \B$ satisfying $\overline\hypd$% be such that $0\le \lb(\phi_1)+\dotsb +\lb(\phi_n)$ and $\ub(\phi_1)+\dotsb +\ub(\phi_n)\le 1$
; $J_p^{\phi_1,\dotsc,\phi_n}(\nA)$ is the set of all $a\in \Sigma(\nA)$ such that there exists a function $u:\mathbb{P}^{n}_+ \to \Delta(\nA)$ for which we have
\[
 a=\int u(t)d\omega(t)
 =\int_{(0,\infty)^n} u(1,t_1,\dotsc,t_n) \, \frac{dt_1}{t_1}\dotsb\frac{dt_n}{t_n}\quad\text{ in }\Sigma(\nA)
\]
 and
\[
 \Phi_p^{\phi_1,\dotsc,\phi_n}\Big(J\big(t,u(t) \big)\Big)<\infty.
\]
This space is equipped with the norm
\[
\left\|a\right\|_{J_p^{\phi_1,\dotsc,\phi_n}(\nA)} =\inf_u \Phi_p^{\phi_1,\dotsc,\phi_n}\Big(J \big(t,u(t)\big)\Big),
\]
the infimum being taken on all $u:\mathbb{P}^{n}_+ \to \Delta(\nA)$ such that $a=\int u(t) \, d\omega(t)$.
\end{Def}
If $(t_1,...,t_n)\in (0,\infty)^n$, we set $\frac{dt}{t}=\frac{dt_1}{t_1}\dotsb \frac{dt_n}{t_n}.$
\begin{Prop}
 Let $p\in [1,\infty]$ and $\phi_1,\dotsc,\phi_n\in \B$ satisfying $\hypd$; the functor $J_p^{\phi_1,\dotsc,\phi_n}$ is an exact interpolation functor of type $f$ where
 \[
  f(t_0,\dotsc,t_n)=t_0\overline{\phi_1}(t_1/t_0)\dotsb \overline{\phi_n}(t_n/t_0).
\]

Moreover,
\[
 \|a\|_{J_p^{\phi_1,\dotsc,\phi_n}(nA)}\le C \frac{1}{\phi_1(t_1)\dotsb\phi_n(t_n)}J(1,t_1,\dotsc,t_n,a).
\]
\end{Prop}
\begin{proof}
 For $a\in\Delta(\nA)$, we have
\[
 a=\frac{1}{(\log 2)^n}\int_{(0,\infty)^n} a\chi_{(1,2)^n}(t) \, \frac{dt}{t},
\]
and
\begin{align*}
 &\|a\|_{J_p^{\phi_1,\dotsc,\phi_n}(\nA)} \\
 &\le \Big(\int_{(0,\infty)^n}\big(\frac{1}{\phi_1(t_1)\dotsb \phi_n(t_n)}J(1,t_1,\dotsc,t_n,a\frac{\chi_{(1,2)^n}(t)}{(\log 2)^n})\big)^p \, \frac{dt}{t}\Big)^{1/p}\\
 &\le \Big(\int_{(0,\infty)^n}\big(\frac{1}{\phi_1(t_1)\dotsb \phi_n(t_n)(\log 2)^n}\max\{1,\frac{t_1}{s_1},\dotsc, \frac{t_n}{s_n} \}  \\
 &\qquad J(1,s_1,\dotsc,s_n,a)\big)^p \, \frac{dt}{t}\Big)^{1/p} \\
&\le C \frac{1}{\phi_1(s_1)\dotsb \phi_n(s_n)}J(1,s_1,\dotsc,s_n,a).
\end{align*}
In particular, by taking $s_j=1$ in the previous inequality, one gets 
\[
 \|a\|_{J_p^{\phi_1,\dotsc,\phi_n}(\nA)}\le C\|a\|_{\Delta(\nA)},
\]
so that $\Delta(\nA)\subsetc J_p^{\phi_1,\dotsc,\phi_n}(\nA)$. Moreover, with the inclusion of the equivalence theorem, one has
\[
 J_p^{\phi_1,\dotsc,\phi_n}(\nA)\subsetc K_p^{\phi_1,\dotsc,\phi_n}(\nA)\subsetc \Sigma(\nA).
\]
Now, let $\nA$ and $\nB$ be two couples in $\snA$ and let $T:\nA\to \nB$. If $a=\int u(s) \, d\omega(s)$ belongs to $J_p^{\phi_1,\dotsc,\phi_n}(\nA)$, then $Ta=\int Tu(s) \, d\omega(s)\in J_p^{\phi_1,\dotsc,\phi_n}(\nB)$ and 
\begin{align*}
 &J^{\nB}(1,s_1,\dotsc,s_n,Tu(1,s_1,\dotsc,s_n)) \\
&\le \|T\|_{A_0,B_0}J^{\nA}\big(1,s_1\frac{\|T\|_{A_1,B_1}}{\|T\|_{A_0,B_0}},\dotsc,s_n\frac{\|T\|_{A_n,B_n}}{\|T\|_{A_0,B_0}},u(1,s_1,\dotsc,s_n)\big),
\end{align*}
so that we have, for $\alpha_j=\frac{\|T\|_{A_j,B_j}}{\|T\|_{A_0,B_0}}$,
\begin{align*}
 &\|Ta\|_{J_p^{\phi_1,\dotsc,\phi_n}(\nB)} \\
 &\le \|T\|_{A_0,B_0}\Big(\int_{(0,\infty)^n} \big(\frac{1}{\phi_1(t_1)\dotsb \phi_n(t_n)}J^{\nA}(1,t_1\alpha_1,\dotsc,t_n\alpha_n,u(t))\big)^p \, \frac{dt}{t}\Big)^{1/p}\\
&=\|T\|_{A_0,B_0}\Big(\int_{(0,\infty)^n} \big(\frac{1}{\phi_1(s_1/\alpha_1)\dotsb \phi_n(s_n/\alpha_n)}J^{\nA}(1,s_1,\dotsc,s_n,u(t))\big)^p \, \frac{ds}{s}\Big)^{1/p}\\
&\le \|T\|_{A_0,B_0}\overline{\phi_1}(\alpha_1)\dotsb \overline{\phi_n}(\alpha_n)\|a\|_{J_p^{\phi_1,\dotsc,\phi_n}(\nA)},
\end{align*}
which ends the proof.
\end{proof}

\begin{Rmk}
Similarly to Proposition \ref{Prop Kfct}, a milder assumption can be adopted in the last proposition. Given $\phi_1,\dotsc,\phi_n\in \B$, $p\in[1,\infty]$, we say that condition $\hypt$ is satisfied when
\[
 \int_{(0,\infty)^n}\big(\frac{1}{\phi_1(t_1)\dotsb\phi_n(t_n)}\min\{1,1/t_1,\dotsc,1/t_n\}\big)^{p'} \, \frac{dt_1}{t_1}\dotsb \frac{dt_n}{t_n}<\infty,
\]
where $p'$ denotes the conjugate exponent of $p$. The last proposition remains valid upon replacing $\hypd$ by $\hypt$. Moreover, note that $\hypd$ implies $\hypt$ and if $p=1$, $\overline\hypd$ implies $\hypt$.
\end{Rmk}
 
By following the methodology outlined in \cite[Remark 4.6]{Spa:74}, one can readily verify the following equivalent definitions.
\begin{Prop}
Let $p\in [1,\infty]$ and $\phi_1,\dotsc,\phi_n\in \B$ satisfying $\hypd$% be such that $0\le \lb(\phi_1)+\dotsb+\lb(\phi_n)$ and $\ub(\phi_1)+\dotsb+\ub(\phi_n)\le 1$
; an equivalent norm on $K_p^{\phi_1,\dotsc,\phi_n}(\nA)$ is given by
\[
 \inf \big(\int \|a_0(t)\|_{A_0}^p d\omega(t)\big)^{1/p}+\sum_{j=1}^n \big(\int (\frac{t_j\phi_1(t_0)\dotsb\phi_n(t_0)}{t_0\phi_1(t_1)\dotsb\phi_n(t_n)}\|a_j(t)\|_{A_j})^p \, d\omega(t)\big)^{1/p},
\]
where the infimum is taken over all decompositions $a=a_0(t)+\dotsb +a_n(t)$, $a_j(t)\in A_j$.
\end{Prop}
\begin{Prop}
Let $p\in [1,\infty]$ and $\phi_1,\dotsc,\phi_n\in \B$ satisfying $\hypd$% be such that $0\le \lb(\phi_1)+\dotsb+\lb(\phi_n)$ and $\ub(\phi_1)+\dotsb+\ub(\phi_n)\le 1$
; an equivalent norm on $J_p^{\phi_1,\dotsb,\phi_n}(\nA)$ is given by
\begin{align*}
 \inf_u\max \{ &
 \big(\int \|u(t)\|_{A_0}^p \, d\omega(t)\big)^{1/p}, \\
 & \max_{1\le j\le n} \big(\int(\frac{t_j\phi_1(t_0)\dotsb \phi_n(t_0)}{t_0\phi_1(t_1)\dotsb \phi_n(t_n)}\|u(t)\|_{A_j})^p \, d\omega(t)\big)^{1/p}\},
\end{align*}
where the infimum is taken over all $u:\mathbb{P}^n_+ \to \Delta(\nA)$ such that $a=\int u(t) \, d\omega(t)$ in $\Sigma(\nA)$.
 \end{Prop}

\section{On the equivalence theorem}\label{sec:equiv}
The results presented in this section could readily be derived from the theory developed above \cite{Lamby:th}; however, since more general results are available \cite{Ase:97}, we shall instead apply those within our framework.

For $n>1$, the $K$- and the $J$- methods are not equivalent, with the $K$-method being more general.
% \[J_p^{\phi_1,\dotsc,\phi_n}(\nA)\subsetc K_p^{\phi_1,\dotsc,\phi_n}(\nA).\]
\begin{Prop}
Given $p\in [1,\infty]$, if $\phi_1,\dotsc,\phi_n\in \B$ satisfy $\hypd$% be such that $0\le \lb(\phi_1)+\dotsb+\lb(\phi_n)$ and $\ub(\phi_1)+\dotsb+\ub(\phi_n)\le 1$
, then
\[
 J_p^{\phi_1,\dotsc,\phi_n}(\nA)\subsetc K_p^{\phi_1,\dotsc,\phi_n}(\nA).
\]
\end{Prop}
 \begin{proof}
The arguments provided in the proof of Theorem 1 in \cite{Ase:97} are sufficient to conclude.
%Let $a\in J_p^{\phi_1,\dotsc,\phi_n}(\nA)$, so there exists $u:\mathbb{P}^{n}_+ \to \Delta(\nA)$ such that
% \[
% a=\int_{(0,\infty)^n} u(1,t_1,\dotsc,t_n) \, \frac{dt_1}{t_1}\dotsb\frac{dt_n}{t_n}.
%\]
%We have
%\begin{align*}
% &K(1,t_1,\dotsc,t_n,a) \\
% &\le \int_{\R^n_+} K\big(1,t_1,\dotsc,t_n,u(1,s_1,\dotsc,s_n)\big) \, \frac{ds_1}{s_1}\dotsb \frac{ds_n}{s_n}\\
% &\le \int_{\R^n_+}\min\{1,\frac{t_1}{s_1},\dotsc,\frac{t_n}{s_n}\}J\big(1,s_1,\dotsc,s_n,u(1,s_1,\dotsc,s_n)\big) \, \frac{ds_1}{s_1}\dotsb\frac{ds_n}{s_n}\\
% &\le \int_{\R^n_+}\min\{1,\frac{1}{s_1},\dotsc,\frac{1}{s_n} \} \\
% &\qquad J\big(1,s_1t_1,\dotsc,s_nt_n,u(1,s_1t_1,\dotsc,s_nt_n)\big) \, \frac{ds_1}{s_1}\dotsb\frac{ds_n}{s_n}.
%\end{align*}
%By properties of Boyd functions and Young's inequality, we can conclude.
\end{proof}
 
Now, we provide a counterexample for the reverse inclusion, inspired by the one used in the classical case (\cite{Spa:74, Yos:70}):
\begin{Ex}
Let $A_0$ and $A_1$ be two Banach spaces and $\phi\in\B$ with $0< \lb(\phi)\le \ub(\phi)< 1$, such that
\[
 A_1\subsetneq K_\infty^{\phi}(A_0,A_1)\subset A_0.
\]
Let $a\in K_\infty^{\phi}(A_0,A_1)\backslash A_1$ and let $A_2$ denote the (one dimensional) vector space generated by $a$ with norm inherited from $A_0$. %$\|\cdot\|_{A_0}$.
We then consider the triplet $\nA=(A_0,A_1,A_2)$. It is clear that $\Delta(\nA)=0$, so that $J_p^{\phi_1,\phi_2}(\nA)=0$ for all $p$, $\phi_1$, and $\phi_2$.

We can construct $\phi_1,\phi_2\in\B$ such that $a\in K_p^{\phi_1,\phi_2}(\nA)$. Without loss of generality, one can assume that
\[
 \max\{\|a\|_{A_0},\|a\|_{K_\infty^{\phi}(A_0,A_1)}\}\le 1.
\]
By the given assumptions, we obtain
\begin{align*}
 K(t,a)&\le \min \{t_0\|a\|_{A_0},K^{(A_0,A_1)}\big((t_0,t_1),a\big),t_2\|a\|_{A_2}\}\\
 &\le \min \{t_0,\frac{t_0\phi(t_1)}{\phi(t_0)},t_2\}.
\end{align*}
Therefore, using $(s_1,s_2)=(\phi(t_1),t_2)$, one has
 \begin{align*}
 \|a\|_{K_p^{\phi_1,\phi_2}(\nA)}
 &\le \big(\int_{\mathbb{P}^{2}_+}(\frac{\phi_1(t_0)\phi_2(t_0)}{t_0\phi_1(t_1)\phi_2(t_2)}\min \{t_0,\frac{t_0\phi(t_1)}{\phi(t_0)},t_2\} )^p \,  dw(t)\big)^{1/p} \\
 &= (\int_{(0,\infty)^2} (\frac{1}{\phi_1(t_1)\phi_2(t_2)}\min \{1,\phi(t_1),t_2\} )^p \, \frac{dt_1}{t_1}\frac{dt_2}{t_2})^{1/p}\\
 &\le C (\int_{(0,\infty)^2}(\frac{1}{\phi_1(\xi(s_1))\phi_2(s_2)}\min \{1,s_1,s_2\})^p \, \frac{ds_1}{s_1}\frac{ds_2}{s_2})^{1/p},
\end{align*}
where $\xi=\phi^{-1}$. Thus, for $\ub(\phi_1)/\lb(\phi)+\ub(\phi_2)<1$, we have $a\in K_p^{\phi_1,\phi_2}(\nA)$.
\end{Ex}
 
We introduce a sufficient condition for the equivalence theorem to hold.
\begin{Def}
We denote by $\sigma(\nA)$ the set of all $a\in \Sigma(\nA)$ for which
\[
 \int \frac{K(t,a)}{\max t} \, d\omega (t)<\infty.
\]
\end{Def}
 
\begin{Def}
The condition $\F(\nA)$ is satisfied if, for every $a\in \sigma(\nA)$, there exists a function $u:\mathbb{P}^{n}_+ \to \Delta(\nA)$ such that
 \[
 a=\int u(t) \, d\omega(t)\quad\text{in }\Sigma(\nA)
\]
and
\[
 J\big(t,u(t)\big)\le C(\nA)K(t,a).
\]
\end{Def}
In the case $n=1$, $\F(\nA)$ is satisfied for any Banach couple $\nA$ \cite[Remark 5.3]{Spa:74}; for $n>1$ however, it is not necessarily the case.
\begin{Thm}\label{equivalence theorem n qcq}
Let $p\in [1,\infty]$ and $\phi_1,\dotsc,\phi_n\in \B$ satisfying $\hypd$% such that $0\le \lb(\phi_1)+\dotsb+\lb(\phi_n)$ and $\ub(\phi_1)+\dotsb+\ub(\phi_n)\le 1$
; if $\F(\nA)$ is satisfied, then
\begin{equation}\label{eq:spc eq}
 J_p^{\phi_1,\dotsc,\phi_n}(\nA)= K_p^{\phi_1,\dotsc,\phi_n}(\nA).
\end{equation}
\end{Thm}
\begin{proof}
This follows from Theorem~1 in \cite{Ase:97} upon taking (\ref{eq:Phip}) as the underlying Banach function lattice.
\end{proof}
If the condition is satisfied, we denote the space \eqref{eq:spc eq} by $\nA^{\phi_1,\dotsc,\phi_n}_p$.
 
%\begin{Def}
%% Let $nA$ and $\nB$,\mathbf B$ in $\mathscr{B}_n$. 
%We say that $\nB$ is a retract of $\nA$ if there exist morphisms $P$ and $I$ such that the diagram
% \begin{center}
%\begin{tikzcd}[column sep=small]
%\nB \arrow["\text{id}"]{rr}\arrow[dr,"I"] & & \nB\\
%& \nA  \arrow[ur,"P"] & 
%\end{tikzcd}
%\end{center}
%is commutative, i.e.\ $P\circ I= \text{\rm id}$.
%\end{Def}
% 
\begin{Rmk}\label{rmk lim cases n>1}
If one wishes to impose a weaker assumption than $\hypd$ (given that $\F(\nA)$ is known to be optimal), one must consider conditions involving the Boyd functions and their derivatives in order to establish a modified fundamental lemma. We do not provide the details here, but the adaptation from \cite{Lamby:25} is straightforward.
\end{Rmk}
\begin{Prop}
 If $\F(\nA)$ is satisfied and $\nB$ is a retract of $\nA$, then $\F(\nB)$ is also satisfied.
\end{Prop}
\begin{proof}
 Let $b\in\sigma(\nB)$, so that $Ib\in \sigma(\nA)$. Therefore, there exists $u:\mathbb{P}^{n}_+ \to \Delta(\nA)$ such that
 \[
 Ib=\int u(t) \, d\omega(t)\quad\text{in }\Sigma(\nA)\]
 and
 \[
 J\big(t,u(t)\big)\le C(\nA)K(t,Ib).
\]
We have 
\[
 b=PIb=\int Pu(t) \, d\omega(t)\quad\text{in } \Sigma(\nB)
\]
 and
 \[
 J^{\nB}\big(t,Pu(t)\big)
 \le C J^{\nA}\big(t,u(t)\big)
 \le C K^{\nA}(t,Ib)\le CK^{\nB}(t,b),
\]
bringing the proof to a close.
\end{proof}
 
By interpolation, we have the following result:
\begin{Prop}
If $\nB$ is a retract of $\nA$, then $K_p^{\phi_1,\dotsc,\phi_n}(\nB)$ is a retract of $K_p^{\phi_1,\dotsc,\phi_n}(\nA)$ and $J_p^{\phi_1,\dotsc,\phi_n}(\nB)$ is a retract of $J_p^{\phi_1,\dotsc,\phi_n}(\nA)$, with the same $P$ and $I$.
\end{Prop}
 
\begin{Thm}\label{thm red to 1-fct}
Let $1\le m\le n-1$, $0< \lambda_1+\dotsb+\lambda_m< 1$, $\eta_0,\eta_{m+1},\dotsc,\eta_n\in \B$ be such that $0<\lb(\eta_0)+\lb(\eta_{m+1})+\dotsb +\lb(\eta_n)$ and $\ub(\eta_0)+\ub(\eta_{m+1})+\dotsb +\ub(\eta_n)<1$. Set
\[
 (\phi_1,\dotsc,\phi_m,\phi_{m+1},\dotsc,\phi_n)= (\eta_0^{\lambda_1},\dotsc,\eta_0^{\lambda_m},\eta_{m+1},\dotsc,\eta_m)
\]
and let $\nB^{(j)}=(A_j,A_{m+1},\dotsc,A_{n})$ for $j\in\{0,\dotsc,m\}$.
If
\[
 K_{p}^{\eta_0,\eta_{m+1},\dotsc,\eta_n}(\nB^{(j)})= J_{p}^{\eta_0,\eta_{m+1},\dotsc,\eta_n}(\nB^{(j)})
\]
(this is the case if $\F(\nB^{(j)})$ is satisfied) for $j\in\{0,\dotsc ,m\}$ and if 
\[
 K^{\phi_1,\dotsc,\phi_n}_p(\nA)= J^{\phi_1,\dotsc,\phi_n}_p(\nA)
\]
(this is the case if $\F(\nA)$ is satisfied), then 
\[
 \big((\nB^{(0)})_{p}^{\eta_0,\eta_{m+1},\dotsc,\eta_n},\dotsc,(\nB^{(m)})_{p}^{\eta_0,\eta_{m+1},\dotsc,\eta_n}\big)^{\lambda_1,\dotsc,\lambda_m}_p= \nA^{\phi_1,\dotsc,\phi_n}_p.
\]
\end{Thm}
\begin{proof}
This follows as a direct consequence of the reiteration theorem in \cite{Ase:97}, with $\Phi$ defined by (\ref{eq:Phip}).
\end{proof}
 
 \section{Some properties}
 From now on, we assume that condition $\hypd$ is always satisfied.
The following property is a direct consequence of H\"older's inequality.
\begin{Prop}
For $1\le p\le q\le \infty$, we have
\[
 K_p^{\phi_1,\dotsc,\phi_n}(\nA)\subsetc K_q^{\phi_1,\dotsc,\phi_n}(\nA)
\]
and
\[
 J_p^{\phi_1,\dotsc,\phi_n}(\nA)\subsetc J_q^{\phi_1,\dotsc,\phi_n}(\nA).
\]
\end{Prop}
Given $\phi\in \ban$, let us set $\phi_q(t)= \phi(t^q)^{1/q}$, with the usual modification in the case $q=\infty$. We directly get the following result:
\begin{Prop}
Let $\pi$ be a permutation of $\{0,1,\dotsc,n\}$ and set
\[
 \nA^{(\pi)}=(A_{\pi(0)},A_{\pi(1)},\dotsc,A_{\pi(n)}),
\]
$\phi_0(t)=t/(\phi_1(t)\dotsb \phi_n(t))$ and
\[
 (\psi_0,\psi_1,\dotsc,\psi_n)
 =\big((\phi_{\pi(0)})_p,(\phi_{\pi(1)})_p,\dotsc,(\phi_{\pi(n)})_p\big).
\] 
We have
\[
 K^{\phi_1,\dotsc,\phi_n}_p(\nA)=K^{\psi_1,\dotsc,\psi_n}_p(\nA^{(\pi)})
\]
and 
\[
 J^{\phi_1,\dotsc,\phi_n}_p(\nA)= J^{\psi_1,\dotsc,\psi_n}_p(\nA^{(\pi)}).
\]
 \end{Prop}

\begin{Prop}\label{reduction to lower order}
If $A_{n-1}=A_n$ ($n>1$), then 
\[
 K^{\phi_1,\dotsc,\phi_n}_p(\nA)=K^{\phi_1,\dotsc,\phi_{n-1}\phi_n}_p(A_0,\dotsc,A_{n-1})
\]
and 
\[
 J^{\phi_1,\dotsc,\phi_n}_p(\nA) =J^{\phi_1,\dotsc,\phi_{n-1}\phi_n}_p(A_0,\dotsc,A_{n-1}).
\]
\end{Prop}
\begin{proof}
Let us show the result for the case $K$. For $\nA'=(A_0,\dotsc,A_{n-1})$, we have
 \[
 K^{\nA}(1,t_1,\dotsc,t_n,a)=K^{\nA'}(1,t_1,\dotsc,\min\{t_{n-1},t_n\},a).
\]
This identity allows to conclude, since
\begin{align*}
&\int_{(0,\infty)^2}\big(\frac{1}{\phi_{n-1}(t_{n-1})\phi_n(t_n)}K^{\nA}(1,t_1,\dotsc,t_n,a)\big)^p \, \frac{dt_{n-1}}{t_{n-1}}\frac{dt_n}{t_n} \\
&= C\int_0^\infty \big(\frac{1}{\phi_{n-1}(s)\phi_n(s)}K^{\nA'}(1,t_1,\dotsc,s,a)\big)^p \, \frac{ds}{s}.
\end{align*}
\end{proof}
 
Let us highlight the following outcome from \cite{Spa:74}.
 \begin{Prop}
If $A_{n-1}=A_n$ and $\F(A_0,\dotsc,A_{n-1})$ is satisfied, then $\F(\nA)$ is also satisfied.
\end{Prop}

\section{Power theorem}\label{sec:powth}
Let us remind ourselves that if $\|\cdot\|$ is a quasi-norm on $A$, then $\|\cdot\|^\alpha$ is also a quasi-norm for $\alpha>0$. We denote the space equipped with this functional by $A^{(\alpha)}$. 
If $\nA=(A_0,\dotsc,A_n)$ is in $\snA$, we naturally define
\[
 \nA^{(p)}=(A_0^{(p)},\dotsc,A_n^{(p)}).
\]
We also set $\phi_p(t)=\phi(t^p)^{1/p}$ for all $t>0$.
\begin{Thm}
Let $p\in [1,\infty]$, $q>0$ and $\phi_1,\dotsc,\phi_n\in \B$ satisfying $\hypd$% be such that $0\le \lb(\phi_1)+\dotsb+\lb(\phi_n)$ and $\ub(\phi_1)+\dotsb+\ub(\phi_n)\le 1$
; we have
\[
 K^{\phi_1,\dotsc,\phi_n}_p(\nA^{(q)})=K^{(\phi_1)_q,\dotsc,(\phi_n)_q}_{pq}(\nA)^{(q)}.\]
\end{Thm}
\begin{proof}
Let us define
\[
 K_\infty(t,a)=
 \inf_a \max \{t_0\|a_0\|_{A_0},t_1\|a_1\|_{A_1},\dotsc,t_n\|a_n\|_{A_n}\},
\]
where the infimum is taken over all the decompositions $a=a_0+\dotsb+a_n$, with $a_j\in A_j$. It can be observed that
\[
 K_\infty(t,a)\le K(t,a)\le n K_\infty(t,a).
\]
If $p=\infty$, we have
\begin{align*}
 \|a\|_{K^{\phi_1,\dotsc,\phi_n}_\infty(\nA^{(q)})}
 &\asymp \sup_s \frac{1}{\phi_1(s_1)}\dotsb\frac{1}{\phi_n(s_n)}K_\infty(1,s_1,\dotsc,s_n,a;\nA^{(q)})\\
&= \sup_t \frac{1}{(\phi_1)_q(t_1)^q}\dotsb\frac{1}{(\phi_n)_q(t_n)^q}K_\infty(1,t_1,\dotsc,t_n,a;\nA)^q.
\end{align*}
If $p<\infty$, we get
\begin{align*}
 &\|a\|_{K^{\phi_1,\dotsc,\phi_n}_p(\nA^{(q)}) }^p \\
 &\asymp \int_{(0,\infty)^n}\big(\frac{1}{\phi_1(s_1)}\dotsb\frac{1}{\phi_n(s_n)} K_\infty(1,s_1,\dotsc,s_n)\big)^p \, \frac{ds_1}{s_1}\dotsb \frac{ds_n}{s_n}\\
 &= \int_{(0,\infty)^n}\big(\frac{1}{(\phi_1)_q(t_1)}\dotsb \frac{1}{(\phi_n)_q(t_n)}K_\infty(1,t_1,\dotsc,t_n)\big)^{pq} \, \frac{dt_1}{t_1}\dotsb \frac{dt_n}{t_n},
\end{align*}
thereby concluding the proof.
\end{proof}

\section{Stability}
In this section, we use $X$ to denote an intermediate space with respect to $\nA$.
\begin{Def}
Let $\phi_1,\dotsc,\phi_n\in \B$ satisfying $\overline\hypd$, %We define the following classes for $X$:
\begin{enumerate}
\item\label{en1:1} $X$ is of class $\clas_K(\phi_1,\dotsc,\phi_n;\nA)$ if \[K(1,t_1,\dotsc,t_n,a)\le C \phi_1(t_1)\dotsb \phi_n(t_n)\|a\|_X,
\]
for all $a\in X$;
\item\label{en1:2} $X$ is of class $\clas_J(\phi_1,\dotsc,\phi_n;\nA)$ if
\[
 \|a\|_X\le C \frac{1}{\phi_1(t_1)\dotsb \phi_n(t_n)}J(1,t_1,\dotsc,t_n,a),
\]
for all $a\in \Delta(\nA)$.
\end{enumerate}
We say that $X$ is of class $\clas(\phi_1,\dotsc,\phi_n;\nA)$ if $X$ is of class $\clas_K(\phi_1,\dotsc,\phi_n;\nA)$ and of class $\clas_J(\phi_1,\dotsc,\phi_n;\nA)$.
\end{Def}

The two following propositions are straightforward from these definitions.
\begin{Prop}
$X$ is of class $\clas_K(\phi_1,\dotsc,\phi_n;\nA)$ if and only if
\[
 X\subsetc K^{\phi_1,\dotsc,\phi_n}_\infty(\nA).
\]
\end{Prop}
\begin{Prop}
$X$ is of class $\clas_J(\phi_1,\dotsc,\phi_n;\nA)$ if and only if
\[
 K^{\phi_1,\dotsc,\phi_n}_1(\nA)\subsetc X.
\]
\end{Prop}

%Using the notation $\lambda_1,\dotsc,\lambda_m$ for $\cdot^{\lambda_1},\dotsc,\cdot^{\lambda_m}$,
We have the following result.
\begin{Prop}\label{pro:gen Rn}
Let $p\in [1,\infty]$, $m\ge n$, $0< \lambda_1+\dotsb +\lambda_m<1$ and $\phi_1^{(j)},\dotsc,\phi_n^{(j)}\in \B$ satisfying $\hypd$ %such that $0\le \lb(\phi_1^{(j)})+\dotsb+\lb(\phi_n^{(j)})$ and $\ub(\phi_1^{(j)})+\dotsb +\ub(\phi_n^{(j)})\le 1$
for all $j\in\{1,\dotsc,m\}$. Set 
\[
 (\phi_1,\dotsc,\phi_n)= ((\phi_1^{(1)})^{\lambda_1}\dotsb (\phi_1^{(m)})^{\lambda_m},\dotsc,(\phi_n^{(1)})^{\lambda_1}\dotsb (\phi_n^{(m)})^{\lambda_m})
\]
and suppose that the vectors 
\[
\begin{pmatrix}
\frac{t_1D\phi_1^{(j)}(t_1)}{\phi_1^{(j)}(t_1)}\\
\vdots\\
\frac{t_nD\phi_n^{(j)}(t_n)}{\phi_n^{(j)}(t_n)}
\end{pmatrix}
\quad (j\in\{1,\dotsc,m\})
\]
generate $\R^{n}$. Finally, let $\mathbf X = (X_0,\dotsc,X_m)$. 
\begin{enumerate}
\item If $X_j$ is of class $\clas_K(\phi_1^{(j)},\dotsc,\phi_n^{(j)};\nA)$ for all $j\in\{0,\dotsc,m\}$, then, with obvious notations,
\[
 K_p^{\lambda_1,\dotsc,\lambda_m}(\mathbf X)\subsetc K^{\phi_1,\dotsc,\phi_n}_p(\nA).
\]
\item If $X_j$ is of class $\clas_J(\phi_1^{(j)},\dotsc,\phi_n^{(j)};\nA)$ for all $j\in\{0,\dotsc,m\}$, then 
\[
 J^{\phi_1,\dotsc,\phi_n}_p(\nA) \subsetc J_p^{\lambda_1,\dotsc,\lambda_m}(\mathbf X).
\]
\end{enumerate}
\end{Prop}
\begin{proof}
First suppose that $m=n$ and remark that 
\[
\left\{
\begin{array}{l}
s_1=\phi_1^{(1)}(t_1) \dotsb\phi_n^{(1)}(t_n)\\
\hphantom{s_1:}\vdots\\
s_n=\phi_1^{(n)}(t_1)\dotsb \phi_n^{(n)}(t_n)
\end{array}
\right.
\]
is an appropriate change of variable if and only if the vectors 
\[
\begin{pmatrix}
\frac{t_1D\phi_1^{(j)}(t_1)}{\phi_1^{(j)}(t_1)}\\
\vdots\\
\frac{t_nD\phi_n^{(j)}(t_n)}{\phi_n^{(j)}(t_n)}
\end{pmatrix}
\quad (j\in\{1,\dotsc,n\})
\]
are linearly independent.

For $a=a_0+\dotsb+a_n\in \Sigma(\mathbf X)$, we get
\begin{align*}
 K(1,t_1,\dotsc,t_n,a;\nA) &\le \|a_0\|_{A_0}+K(t,a_1;\nA) + \dotsb + K(t,a_n;\nA)\\
&\le C (\|a_0\|_{X_0}+\phi_1^{(1)}(t_1)\dotsb \phi_n^{(1)}(t_n)\|a_1\|_{X_1} \\
&\quad +\dotsb +\phi_1^{(n)}(t_1)\dotsb \phi_n^{(n)}(t_n)\|a_n\|_{X_n}).
\end{align*}
Therefore, using
\[
\left\{
\begin{array}{c}
s_1=\phi_1^{(1)}(t_1)\dotsb \phi_n^{(1)}(t_n)\\
\quad \vdots\\
s_n=\phi_1^{(n)}(t_1)\dotsb \phi_n^{(n)}(t_n)
\end{array}
\right.,
\]
one gets (\ref{en1:1}), since
\begin{align*}
 \|a\|_{K^{\phi_1,\dotsc,\phi_n}_p(\nA)}
 &=\Phi_p^{\phi_1,\dotsc,\phi_n} \big(K(t,a)\big) \\
 &\le C\Phi_p^{\lambda_1,\dotsc,\lambda_m}\big(K(s,a;\mathbf X)\big)\\
&= C\|a\|_{K_p^{\lambda_1,\dotsc,\lambda_m}(\mathbf X)}.
\end{align*}
Let us prove (\ref{en1:2}). Set $a=\int u(t)d\omega(t)\in J^{\phi_1,\dotsc,\phi_n}_p(\nA)$ and define $v$ so that
\[
 C(s)v\big(s(t)\big)=u(t).
\]
We have
\begin{align*}
 \|a\|_{J_p^{\lambda_1,\dotsc,\lambda_m}(\mathbf X)}
 &=\Phi_p^{\lambda_1,\dotsc,\lambda_m}(J(s,v(s);\mathbf X)\big)\\
&= C \Phi_p^{\phi_1,\dotsc,\phi_n}\big(J(s(t),v\big(s(t)\big);\mathbf X)\big)\\
&\le C  \Phi_p^{\phi_1,\dotsc,\phi_n}\big(J(t,u(t);\nA)\big).
\end{align*}

Now suppose $m>n$ and consider the $(m+1)$-tuple
\[
 \nA'=(A_0,\dotsc,A_n,A_n,\dotsc,A_n).
\]
Using Proposition~\ref{reduction to lower order}, we get
\[
 K_p^{\psi_1,\dotsc,\psi_m}(\nA')=K_p^{\psi_1,\dotsc,\psi_{n-1},\psi_n\dotsc \psi_m}(\nA),
\]
bringing the proof to a close.
\end{proof}
By adjusting the reasoning from the preceding demonstration, we derive the following results.
\begin{Thm}\label{thm pour fig1}
Let $p\in [1,\infty]$, $m\ge n$, $0< \lambda_1+\dotsb +\lambda_m< 1$ and $\phi_1^{(j)},\dotsc ,\phi_n^{(j)}\in \B$ satisfying $\hypd$ %such that $0\le \lb(\phi_1^{(j)})+\dotsb+\lb(\phi_n^{(j)})$ and $\ub(\phi_1^{(j)})+\dotsb +\ub(\phi_n^{(j)})\le 1$
for all $j\in\{1,\dotsb,m\}$. Set
\[
 \big(\phi_1,\dotsc,\phi_n)=\big( (\phi_1^{(1)})^{\lambda_1}\dotsb (\phi_1^{(m)})^{\lambda_m},\dotsc, (\phi_n^{(1)})^{\lambda_1}\dotsb (\phi_n^{(m)})^{\lambda_m})\big)
\]
and suppose that the vectors 
\[
\begin{pmatrix}
\frac{t_1D\phi_1^{(j)}(t_1)}{\phi_1^{(j)}(t_1)}\\
\vdots\\
\frac{t_nD\phi_n^{(j)}(t_n)}{\phi_n^{(j)}(t_n)}
\end{pmatrix}
\quad (j\in\{1,\dotsc,m\})
\]
generate $\R^{n}$. Finally, let $\mathbf X = (X_0,\dotsc,X_m)$. If $X_j$ is of class $\clas(\phi_1^{(j)},\dotsc,\phi_n^{(j)};\nA)$ for all $j\in\{0,\dotsc,m\}$ and if 
\[
 K^{\phi_1,\dotsc,\phi_n}_p(\nA)=J^{\phi_1,\dotsc,\phi_n}_p(\nA)
\]
(this is the case if $\F(\nA)$ is satisfied), then 
\[
 \mathbf X^{\lambda_1,\dotsc,\lambda_m}_p= \mathbf A^{\phi_1,\dotsc,\phi_n}_p.
\]
\end{Thm}

\begin{figure}
\begin{center}
\begin{tikzpicture}[scale=0.6]

% Coordonnées des sommets du triangle isocèle
\coordinate (A0) at (0,0);
\coordinate (A1) at (12,0);
\coordinate (A2) at (6,10);

% Coordonnées des sommets du quadrilatère
\coordinate (X0) at (4,0);
\coordinate (X1) at (10.6,1.2);
\coordinate (X2) at (7,6);
\coordinate (X3) at (3.6,4);
% Calcul du centre de gravité
\coordinate (G) at (barycentric cs:X0=1,X1=1,X2=1,X3=1);

% Dessin du triangle isocèle agrandi
\draw (A0) -- (A1) -- (A2) -- cycle;

% Dessin du quadrilatère
\draw (X0) -- (X1) -- (X2) -- (X3) -- cycle;

% Étiquetage des sommets
\node[below] at (A0) {$A_0$};
\node[below] at (A1) {$A_1$};
\node[above] at (A2) {$A_2$};
\node[below] at (X0) {$X_0$};
\node[below] at (X1) {$X_1$};
\node[right] at (X2) {$X_2$};
\node[above] at (X3) {$X_3$};

% Ajout de la croix (x) au-dessus du signe "="
\node at (G)  {$\times$};
%\node at (6.05, 2.8)  {$\times$};
\node[below] at (G) {$\mathbf{X}_p^{\lambda_1,\lambda_2,\lambda_3}=\nA_p^{\phi_1,\phi_2}$};
%\node at (7,2.2) {$\mathbf{X}_p^{\lambda_1,\lambda_2,\lambda_3}=\nA_p^{\phi_1,\phi_2}$};

\end{tikzpicture}
\end{center}
\caption{Illustration inspired of \cite[Figure 9.1]{Spa:74} of Theorem \ref{thm pour fig1} with $m=3$ and $n=2$.}
\end{figure}
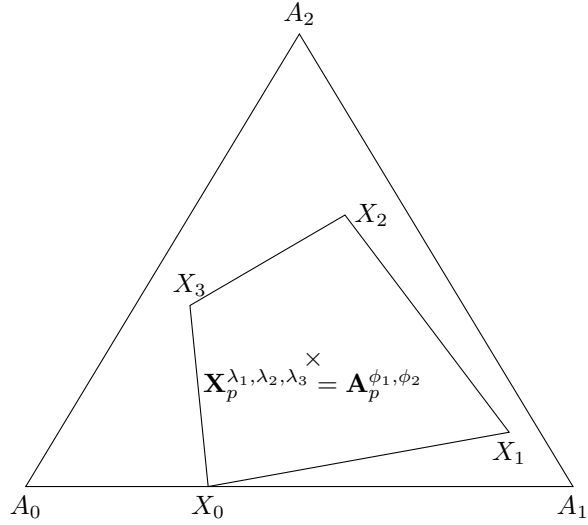

\begin{Cor}
Let $m\ge n$, $p,p_0,\dotsc,p_m\in [1,\infty]$, $0<\lambda_1+\dotsb+\lambda_m<1$ and $\phi_1^{(j)},\dotsc,\phi_n^{(j)}\in \B$ satisfying $\hypd$ %be such that $0\le \lb(\phi_1^{(j)})+\dotsb+\lb(\phi_n^{(j)})$ and $\ub(\phi_1^{(j)})+\dotsb+\ub(\phi_n^{(j)})\le 1$
for all $j\in\{1,\dotsc,m\}$. Set 
\[
 (\phi_1,\dotsc,\phi_n)
 =\big( (\phi_1^{(1)})^{\lambda_1}\dotsb (\phi_1^{(m)})^{\lambda_m}),\dotsc, (\phi_n^{(1)})^{\lambda_1}\dotsb (\phi_n^{(m)})^{\lambda_m})\big)
\]
and suppose that the vectors 
\[
\begin{pmatrix}
\frac{t_1D\phi_1^{(j)}(t_1)}{\phi_1^{(j)}(t_1)}\\
\vdots\\
\frac{t_nD\phi_n^{(j)}(t_n)}{\phi_n^{(j)}(t_n)}
\end{pmatrix}
,\quad j\in\{1,\dotsc,m\}
\]
generate $\R^{n}$. If $K_{p_j}^{\phi_1^{(j)},\dotsc,\phi_n^{(j)}}(\nA)=J_{p_j}^{\phi_1^{(j)},\dotsc,\phi_n^{(j)}}(\mathbf A)$ for $j\in\{0,\dotsc,m\}$ and if 
\[
 K^{\phi_1,\dotsc,\phi_n}_p(\nA)= J^{\phi_1,\dotsc,\phi_n}_p(\nA)
\]
(this is the case if $\F(\nA)$ is satisfied), then 
\[
 (\nA_{p_0}^{\phi_1^{(0)},\dotsc,\phi_n^{(0)}},\dotsc,\nA_{p_m}^{\phi_1^{(m)},\dotsc,\phi_n^{(m)}})^{\lambda_1,\dotsc,\lambda_m}_p=\nA^{\phi_1,\dotsc,\phi_n}_p.
\]
\end{Cor}
\begin{Prop}
Let $m\ge n$, $\phi_1^{(j)},\dotsc,\phi_n^{(j)}\in \B$ be such that $\hypd$ is satisfied %$0\le \lb(\phi_1^{(j)})+\dotsb +\lb(\phi_n^{(j)})$ and $\ub(\phi_1^{(j)})+\dotsb +\ub(\phi_n^{(j)})\le 1$
for all $j\in\{1,\dotsc,m\}$. Suppose that the vectors 
\[
\begin{pmatrix}
\frac{t_1D\phi_1^{(j)}(t_1)}{\phi_1^{(j)}(t_1)}\\
\vdots\\
\frac{t_nD\phi_n^{(j)}(t_n)}{\phi_n^{(j)}(t_n)}
\end{pmatrix}
,\quad j\in\{1,\dotsc,m\}
\]
generate $\R^n$. Finally, let $\mathbf X = (X_0,\dotsc,X_m)$ such that $X_j$ is of class $\clas(\phi_1^{(j)},\dotsc,\phi_n^{(j)};\nA)$ for all $j\in\{0,\dotsc,m\}$. If $\F(\nA)$ is satisfied, then $\F(\mathbf X)$ is also satisfied.
\end{Prop}

Additional stability results can be achieved by imposing restrictions on $\lambda_1,\dotsc,\lambda_m$, $p$, $p_0,\dotsc,p_m$ instead of $\phi_1^{(j)},\dotsc,\phi_n^{(j)}$. The ensuing result can be inferred from \cite{Spa:74}:
\begin{Thm}
Let $p,p_0,\dotsc,p_m\in [1,\infty]$, $0< \lambda_1+\dotsb +\lambda_m< 1$ and $\phi_1^{(j)},\dotsc,\phi_n^{(j)}\in \B$ satisfying $\hypd$ %be such that $0\le \lb(\phi_1^{(j)})+\dotsb +\lb(\phi_n^{(j)})$ and $\ub(\phi_1^{(j)})+\dotsb +\ub(\phi_n^{(j)})\le 1$
for all $j\in\{1,\dotsc,m\}$. Set 
\[
 (\phi_1,\dotsc,\phi_n)= ((\phi_1^{(1)})^{\lambda_1}\dotsb (\phi_1^{(m)})^{\lambda_m}),\dotsc ,(\phi_n^{(1)})^{\lambda_1}\dotsb (\phi_n^{(m)})^{\lambda_m}))
\]
and suppose that 
\[
 \frac{1}{p}= \sum_{j=0}^m \frac{\lambda_j}{p_j}.
\]
If $K_{p_j}^{\phi_1^{(j)},\dotsc,\phi_n^{(j)}}(\nA)=J_{p_j}^{\phi_1^{(j)},\dotsc,\phi_n^{(j)}}(\nA)$ for $j\in\{0,\dotsc,m\}$ and if 
\[
 K^{\phi_1,\dotsc,\phi_n}_p(\nA)=J^{\phi_1,\dotsc,\phi_n}_p(\nA)
\]
(this is the case if $\F(\nA)$ is satisfied), then 
\[
 (\nA_{p_0}^{\phi_1^{(0)},\dotsc,\phi_n^{(0)}},\dotsc,\nA_{p_m}^{\phi_1^{(m)},\dotsc,\phi_n^{(m)}})^{\lambda_1,\dotsc,\lambda_m}_p=\mathbf A^{\phi_1,\dotsc,\phi_n}_p.
\]
\end{Thm}

\section{Interpolation of generalized Sobolev spaces}\label{sec:sobolev}
Let us now turn our attention to the interpolation of multiple generalized Sobolev spaces. First, let us provide a brief introduction to generalized Sobolev and Besov spaces (\cite{Merucci:84,Cobos:88}).
\begin{Def}
Let $q\in[1,\infty]$ and $\phi\in\B$; if $X$ is a Banach space, the space $l^q_\phi(X)$ consists of all sequences $(a_j)_j$ of $X$ such that 
\[
 (\phi(2^j)\|a_j\|_X)_j\in l^q.
\]
This space is equipped with the norm
\[
 \|(a_j)_j\|_{l^q_\phi(X)}
 = \big\|(\phi(2^j)\|a_j \|_X)_j\big\|_{l^q}.
\]
\begin{itemize}
\item If $\phi(t) =t^s$ ($s\in\R$), we set $l^q_s(X) =l^q_\phi(X)$;
\item If $X=\C$, then we set $l^q_\phi = l^q_\phi(X)$.
\end{itemize}
\end{Def}
Let $\Sc'=\Sc'(\R^d)$ be the space of tempered distributions and $\F$ the Fourier transform on  $\Sc'$.
\begin{Def}Let $\J$ be the Bessel operator of order $s$, that is
\[
 \J^s f=\F^{-1}\big((1+|\cdot|^2)^{s/2}\F f\big),
\]
for $f\in\Sc'$ and $s\in\R$. Given $s\in\R$ and $p\in[1,\infty]$, the fractional Sobolev space $H^s_p$ is defined by
\[
 H^s_p= \{f\in \Sc' : \| \J^s f \|_{L^p}<\infty \}.
\]
\end{Def}
We denote by $\B''$ the set of functions $\phi \in\B$ which are $C^{\infty}$ on $[1,\infty)$ and such that 
\[
 t^m|D^m \phi (t)|\le C_m \phi(t),
% t^m|\phi^{(m)}(t)|\le C_m \phi(t),
\]
for all $t\in[1,\infty)$ and all $m\in\N$.

\begin{Ex}
Given $\theta,\gamma\in \R$, let us set
\[
 \phi(t)=t^\theta(1+|\log t|)^\gamma,
\]
for $t>0$; it can be readily verified that $\phi \in\B''$.
\end{Ex}

\begin{Def}
Given $\phi\in \B''$, the generalized Bessel operator $\J^\phi$ is defined by
\[
 \J^\phi f
 = \F^{-1}(\phi(\sqrt{1+|\cdot|^2})\F f),
\]
for $f\in\Sc'$.
\end{Def}

\begin{Rmk}
The previous definition makes sense, since $\phi\in \B''$ and $f\in\Sc'$ imply $\phi(\sqrt{1+|\cdot|^2})\F f\in \Sc'$. Moreover, it is clear that $\J^\phi$ is a linear bijective operator on $\Sc'$ such that $(\J^\phi)^{-1}=\J^{1/\phi}$ and $\J^\phi(\Sc)=\Sc$.
\end{Rmk}

\begin{Def}
Given $\phi\in\B''$ and $p\in[1,\infty]$, the generalized (fractional) Sobolev space $H^\phi_p$ is defined by
\[
 H^\phi_p= \{f\in \Sc' : \|\J^{\phi} f \|_{L^p}<\infty \}.
\]
\end{Def}
\begin{Rmk}
For $\phi(t)=t^s$, it is self-evident that $H^\phi_p=H^s_p$.
\end{Rmk}

\begin{Def}
For $p\in[1,\infty]$, the space of \emph{Fourier multipliers} on $L^p$ is defined by
\[
 M^p= \{f\in\Sc': \forall g\in\Sc,\ \|(\F^{-1} f)* g \|_{L^p}\le C\|g\|_{L^p} \}.
\]
This space is equipped with the norm
\[
 \|f\|_{M^p}
 = \sup \{ \|(\F^{-1} f)* g \|_{L^p}:
 g\in\Sc \text{ and } \|g\|_{L^p}\le 1 \}.
\]
\end{Def}

Let's recall the following classical result \cite{Bergh:76}:
\begin{Lemma}\label{lemme fourier multipliers}
Let $N\in\N$, if $f\in\Sc'$ is such that $f\in L^2$ and $D^\alpha f\in L^2$ for $|\alpha|=N>d/2$, then $f\in M^p$ and 
\[
 \|f\|_{M^p}
 \le C \|f\|_{L^2}^{1-\frac{d}{2N}} (\sup_{|\alpha|=N}\|D^\alpha f\|_{L^2} )^{\frac{d}{2N}},
\]
for some constant $C>0$.
\begin{Def}
Let $\varphi$ be a function in $\Sc$ with support included in the set $\left\{x\in \R^d : \frac{1}{2}\le |x|\le 2\right\}$, such that $\varphi(x)>0$ if $\frac{1}{2}\le |x|\le 2$ and 
\[
 \sum_{j\in\Z}\varphi(2^{-j} x)=1,
\]
for $x\not= 0$. Define functions $\psi_0$ and $\Phi_j$ ($j\in\Z$) in $\Sc$ such that
\[
 \F\Phi_j(x)=\varphi(2^{-j}x)
 \quad\text{ and }\quad\F\psi_0(x)=1-\sum_{j=1}^{\infty}\varphi(2^{-j} x)
\]
and set $\varphi_0=\psi_0$, $\varphi_j=\Phi_j$ for all $j\in\N_0$. We call $(\varphi_j)_j$ a system of test functions.
\end{Def}
\end{Lemma}
\begin{Prop}\label{pro fourier multipliers}
Let $f\in\Sc'$, $\phi\in\B''$ and $p\in[1,\infty]$. If $\varphi_j\ast f$ belongs to $L^p$ for all $j\in\N$, then, for all $j\in\N$,
\[
 \|\J^\phi \varphi_j\ast f \|_{L^p}
 \le C \phi(2^j) \| \varphi_j\ast f \|_{L^p},
\]
for some positive constant $C$.
\end{Prop}
\begin{proof}
Since $\varphi_j* f =\sum_{l=-1}^1\varphi_{j+l} * \varphi_j* f$, we have
\[
 \| \J^\phi \varphi_j* f \|_{L^p}
 \le \sum_{l=-1}^1 \|\F\J^\phi \varphi_{j+l} \|_{M^p} \| \varphi_j* f \|_{L^p}.
\]
Hence, it suffices to demonstrate that
\begin{equation}\label{equa 1 pro avec les fourier multipliers}
 \|\F\J^\phi \varphi_{j+l} \|_{M^p}
 \le C \phi(2^j).
\end{equation}
From
\begin{align*}
 \F\J^\phi \varphi_{j+l}
 =\phi (\sqrt{1+|\cdot|^2}) \F\varphi_{j+l},
\end{align*}
we deduce that (\ref{equa 1 pro avec les fourier multipliers}) is equivalent to 
\begin{equation}\label{equa 2 pro avec les fourier multipliers}
 \|\phi(2^{j+l}\sqrt{2^{-2(j+l)}+|\cdot|^2}) \varphi(\cdot) \|_{M^p}\le C\phi(2^j).
\end{equation}
Now for $\alpha\in\N^d$ and $g=\phi(2^{j+l}\sqrt{2^{-2(j+l)}+|\cdot|^2}) \varphi(\cdot)$, the term $|D^\alpha g(x)|$ is bounded by expressions of the form
\[
 C2^{(j+l)\alpha_1} |\phi^{(\alpha_1)}(2^{j+l}\sqrt{2^{-2(j+l)}+|x|^2})|
\]
in finite number, with $0\le \alpha_1\le |\alpha|$. Since $\phi \in\B''$, $|D^\alpha g(x)|$ is bounded by a finite sum of terms of the form
\[
 C\phi(2^j)\overline{\phi}(\sqrt{2^{-2(j+l)}+|x|^2}) (2^{-2(j+l)}+|x|^2)^{-\alpha_1/2}.
\]
From Lemma~\ref{lemme fourier multipliers} with $g$, the properties of $\overline{\phi}$ and the fact that the support of $\varphi$ is in $\{x\in \R^d : \frac{1}{2}\le |x|\le 2\}$, we obtain
\[
 \|g\|_{M^p}
 \le  C \|g\|_{L^2}^{1-\frac{d}{2N}} (\sup_{|\alpha|=N}\|D^\alpha f\|_{L^2})^{\frac{d}{2N}}
 \le C\phi(2^j),
\]
which means that inequality~(\ref{equa 2 pro avec les fourier multipliers}) is satisfied.
\end{proof}

\begin{Lemma}\label{lemme inclusion des Sobolev gen}
If $\phi\in \B''$ is such that $\lb(\phi)>0$, then $\J^{1/\phi}$ maps $L^p$ into $L^p$ continuously.
\end{Lemma}
\begin{proof}
From Proposition~\ref{pro fourier multipliers}, we get
\begin{align*}
 \|\J^{1/\phi} f \|_{L^p}
 \le \sum_{j} \|\J^{1/\phi} \varphi_j* f \|_{L^p}
 &\le C\sum_{j} \frac{1}{\phi}(2^j) \| \varphi_j* f \|_{L^p}\\
 &\le C \big(\sum_{j}\frac{1}{\phi}(2^j)\big)\|f\|_{L^p}\\
% &\le C (\sum_{j}\int_{2^j}^{2^{j+1}}\frac{1}{\phi}(2^j)\frac{dt}{t})\|f\|_{L^p}\\
 &\le C (\int_{0}^{1}\overline{\phi}(t)\frac{dt}{t}) \|f\|_{L^p}\\
 &\le C\|f\|_{L^p},
\end{align*}
which ends the proof.
\end{proof}

\begin{Prop}\label{pro inclusion des Sobolev gen}
Let $\phi_0,\phi_1\in \B''$, $\phi=\phi_0/\phi_1$, $p\in[1,\infty]$; if $\lb(\phi)>0$, then we have $H^{\phi_0}_p\hookrightarrow H^{\phi_1}_p$.
\end{Prop}
\begin{proof}
From Lemma~\ref{lemme inclusion des Sobolev gen}, we directly get
\[
 \|\J^{\phi_1}f \|_{L^p}
 = \|\J^{1/\phi}\J^{\phi_0}f \|_{L^p}\le C \|\J^{\phi_0}f \|_{L^p},
\]
hence the conclusion.
\end{proof}

\begin{Def}Let $\phi\in\B$ and $(\varphi_k)_k$ be a system of test functions. For $p,q\in[1,\infty]$, the generalized Besov Space is the space 
\[
 B^\phi_{p,q}
 =\{f\in\Sc': (\varphi_k* f)_k\in l^q_\phi(L^p)\}.
\]
This space is equipped with the norm
\[
 \|f\|_{B^\phi_{p,q}}
 = \|(\varphi_k * f)_k \|_{l^q_\phi(L^p)}
 = \|(\|\varphi_k * f \|_{L^p})_k \|_{l^q_\phi}.
\]
For $\phi(t)=t^s$ ($s\in \R$), we set $B^s_{p,q}= B^\phi_{p,q}$, i.e.,
\[
 B^s_{p,q}
 =\{f\in\Sc': (\varphi_k* f)_k\in l^q_s(L^p)\}.
\]
\end{Def}

Let $\gamma\in \B$, $\phi_0,\phi_1\in \B''$ and set $f=\phi_0/\phi_1$, $\psi=\phi_0/(\gamma\circ f)$. For $p,q\in[1,\infty]$, it is well known (see \cite{Merucci:84,Loosveldt:18}) that if $\lb(f)>0$ or $\ub(f)<0$ and if $0<\lb(\gamma)\le \ub(\gamma)<1$, then we have
\[
 K^\gamma_q(H^{\phi_0}_p,H^{\phi_1}_p)
 =B^{\psi}_{p,q}.
\]
Let us generalize this result to several spaces.
\begin{Thm}\label{interpolation de sobolev gen est un besov gen}
Let $p,q\in[1,\infty]$, $\gamma_1,\dotsc,\gamma_n\in \B$, $\phi_0,\phi_1,\dotsc,\phi_n\in \B''$ and for $l_1,l_2\in\{1,\dotsc,n\}$, set $f_{l_1,l_2}=\phi_{l_1}/\phi_{l_2}$. If  $\lb(f_{l_1,l_2})>0$ or $\ub(f_{l_1,l_2})<0$ for $l_1<l_2$ and if $0<\lb(\gamma_k)\le \ub(\gamma_k)<1$, then
\[
 K_q^{\gamma_1,\dotsc,\gamma_n}(H^{\phi_0}_p,\dotsc,H^{\phi_n}_p)=B^{\psi}_{p,q},
\]
with $\psi=\phi_0/(f_{0,1}\circ\gamma_1\dotsb f_{0,n}\circ\gamma_n)$.
\end{Thm}
\begin{proof}
Let us assume that $\lb(f_{l_1,l_2})>0$ for $l_1<l_2$; it is evident that $\psi\in\B$. Moreover, we can suppose that $f_{l_1,l_2}\in\Bp_+$.

Let us first show that
\begin{equation}\label{equa 1 inter sob}
K_q^{\gamma_1,\dotsc,\gamma_n}(H^{\phi_0}_p,\dotsc,H^{\phi_n}_p)=K_q^{\eta_2,\dotsc,\eta_n}(H^{\phi_1}_p,\dotsc,H^{\phi_n}_p),
\end{equation}
with
\[
 \eta_2(t)=\frac{\gamma_1 \Big(f_{0,1}\big( f_{1,2}^{-1}(t) \big) \Big)\gamma_2 \Big(f_{0,2} \big( f_{1,2}^{-1}(t)\big)\Big)}{f_{0,1}\big(f_{1,2}^{-1}(t)\big)}
\]
and
\[
\eta_l(t)=\gamma_l \Big( f_{0,l}\big( f_{1,l}^{-1}(t) \big)\Big),
\]
for $l\in\{3,\dotsc,n\}$. We already know that
\[
 \Sigma(H^{\phi_0}_p,\dotsc,H^{\phi_n}_p)
 =\Sigma(H^{\phi_1}_p,\dotsc,H^{\phi_n}_p).
\]
From there, we can get (\ref{equa 1 inter sob}) using $s_l=f_{1,l}^{-1}(t_l)$ ($l=2,\dotsc,n$), Proposition \ref{reduction to lower order} and
\[
 (u_1,\dotsc,u_n)=(f_{0,1}(s_1),\dotsc,f_{0,n}(s_n)).
\]
Applying the same procedure, we obtain
\begin{equation}\label{equa 2 inter sob}
 K_q^{\gamma_1,\dotsc,\gamma_n}(H^{\phi_0}_p,\dotsc,H^{\phi_n}_p)
=K_q^{\eta}(H^{\phi_{n-1}}_p,H^{\phi_n}_p),
\end{equation}
with $\eta$ such that $\psi=\phi_{n-1}/(\eta\circ f_{(n-1),n})$. Given that the right-hand side of \eqref{equa 2 inter sob} is equal to $B^{\psi}_{p,q}$, we can draw the conclusion.
\end{proof}
In the case where $\phi_j(t) =t^{s_j}$ ($s_j\in\R$) for $j\in \{0,\dotsc,n\}$, we recover a result of \cite{Spa:74}:
\[
 K_q^{\theta_1,\dotsc, \theta_n}(H_p^{s_0},\dotsc, H_p^{s_n}) = B_{p,q}^s,
\]
with $s=(1-\sum_{j=1}^n \theta_j) s_0+ \sum_{j=1}^n \theta_j s_j$, where $\theta_1,\dotsc,\theta_n >0$ are such that $\sum_{j=1}^n \theta_j<1$.

\section{Extension of the Stein-Weiss Theorem}\label{sec:lorentz}
To illustrate the applicability of the theory introduced here, we present a result concerning the generalized interpolation of two spaces, in a context where interpolation across three spaces becomes necessary.

We begin by recalling the definition of the generalized Lorentz spaces \cite{Lorentz:50}:
\begin{Def}
Let $(\Omega,\mu)$ be a measure space with totally $\sigma$-finite measure $\mu$ and let $\omega$ be a fixed weight function on $\Omega$, i.e., a positive measurable function on $\Omega$. Given $\phi\in\B$ and $p\in [1,\infty]$, the Lorentz space $\Lambda_p^\phi(\omega d\mu)$ consists of all $\mu$-measurable functions $a$ on $\Omega$ for which
\[
 \|a\|_{\Lambda_p^\phi(\omega d\mu)}
  =\|\phi\, (\omega a)^*_\mu\|_{L^p_*}
 \] 
is finite, where $(a)^*_\mu$ denotes the non-increasing rearrangement of $|a|$ on $(0,\infty)$ with respect to $\mu$. Specifically, defining
\[
S_a(\alpha) =\mu(\{x\in \Omega: |a(x)|>\alpha\}),
\]
we set
\[
 (a)_\mu^*(t) = \inf\{ \alpha\ge 0: S_a(\alpha)\le t\},
\]
for $t\ge 0$. 
\end{Def}
The space $\Lambda_p^\phi(\omega d\mu)$ is quasi-normed with respect to $\|\cdot\|_{\Lambda_p^\phi}$.
\begin{Ex}
When $\phi(t)=t^{1/q}$, one recovers the usual Lorentz-Zygmund space (see, e.g., \cite{Bergh:76}):
\[
 \Lambda_p^\phi(d\mu)=L^{q,p}(d\mu).
\]
\end{Ex}

We now consider the generalized interpolation of two such spaces. On the one hand, the classical Stein-Weiss theorem can be stated as follows:
\begin{Thm}%\emph{(Stein-Weiss Theorem)}
Let $p_0, p_1 \in [1, \infty]$, $\theta \in (0,1)$, and suppose $\frac{1}{p} = \frac{1 - \theta}{p_0} + \frac{\theta}{p_1}$. Define
\[
 \omega = \omega_0^{(1 - \theta)p/p_0} \, \omega_1^{\theta p / p_1}.
\]
Then,
\begin{equation}\label{equa Stein Weiss}
 K^\theta_p(L^{p_0}(\omega_0 d\mu), L^{p_1}(\omega_1 d\mu)) = L^p(\omega d\mu).
\end{equation}
\end{Thm}
On the other hand, in the unweighted setting, we have the following result for the generalized case \cite{Merucci:84}:
\begin{Thm}\label{interpolation de Lorentz gen est un Lorentz gen}
Let $\gamma,\phi_0,\phi_1\in\B$ and define $f=\phi_0/\phi_1$, $\psi=\phi_0/(\gamma\circ f)$. Let $p,p_0,p_1,q\in[1,\infty]$. If $\lb(f)>0$ or $\ub(f)<0$ and if $0<\lb(\gamma)$, $\ub(\gamma)<1$, then
\begin{equation}\label{Equation inter Lorentz non weighted}
 K^\gamma_q(\Lambda_{p_0}^{\phi_0}(d\mu),\Lambda_{p_1}^{\phi_1}(d\mu)) = \Lambda_q^{\psi}(d\mu).
\end{equation}
\end{Thm}

Combining (\ref{Equation inter Lorentz non weighted}) and (\ref{equa Stein Weiss}), one might naturally expect a relation of the form
\begin{equation}\label{equa fausse}
 K^\gamma_q(\Lambda_{p}^{\phi_0}(\omega_0 d\mu),\Lambda_{p}^{\phi_1}(\omega_1 d\mu)) = \Lambda_q^{\psi}(\omega d\mu),
\end{equation}
for a suitably chosen $\omega$. However, this identity (\ref{equa fausse}) is generally false (see \cite{Ase:01} for a counterexample). In fact, even in the classical setting, interpolation of two Lorentz spaces typically yields a so-called block-Lorentz space \cite{Ase:01}. Strikingly, proving such a result requires the use of reiteration formulas involving three spaces.

We now show how this can be extended within the framework developed here.
\begin{Def}
Let $\phi\in\B$, $p\in [1,\infty]$ and $\omega$ be a weight on $\Omega$; we consider the family of quasi-Banach function lattices $L^\phi_p(\omega d\mu)$ defined by the quasi-norm
\[
 \|a\|_{L^\phi_p(\omega d\mu)}
  =(\int_\Omega |\phi(a(x)\omega(x))|^p d\mu(x))^{1/p},
\]
with the usual modification if $p=\infty$.
\end{Def}
We will need the following ``intermediate'' spaces.
\begin{Def}
Let $\phi\in\B$, $p,q\in [1,\infty]$ and let $\omega$ be a weight on $\Omega$; the block-Lorentz space $L^{\phi}_{p,q}(\omega d\mu)$ is defined as the space of measurable functions $a$ for which the quasi-norm
\[
 \left(\sum_{k\in \Z}(\|a\chi_{\Omega^{\omega}_k}\|_{\Lambda^p_{\phi}(\omega d
 \mu)})^q\right)^{1/q}
\]
is finite, with the usual modification if $q=\infty$. Here, the level sets $\Omega^{\omega}_k$ are defined by
\[
 \Omega^{\omega}_k= \{x\in \Omega : 2^k\le \omega (x) < 2^{k+1}\}.
\]
\end{Def}
In general, the interpolation of two spaces of the form $L^\phi_p(\omega d\mu)$ does not yield a block-Lorentz space \cite{Ase:01}. However, a more favorable situation arises when interpolating three such spaces.
\begin{Lemma}\label{lem:lor}
Let $p_0,p_1,p_2,q\in [1,\infty]$, $\phi_0,\phi_1,\phi_2\in\B$ and $\theta_1,\theta_2\in (0,1)$. Suppose that, for all $t>0$, the points $(\frac{tD\phi_j(t)}{\phi_j(t)},\frac{1}{p_j})$, $j=0,1,2$, are not collinear. Then
\[
 K^{\theta_1,\theta_2}_q (L^{\phi_0}_{p_0}(\omega d\mu),L^{\phi_1}_{p_1}(\omega d\mu),L^{\phi_2}_{p_2}(\omega d\mu))
  = L^{\psi}_{p,q}(\omega d\mu),
\]
where $1/p=\sum_{j=0}^2\theta_j/p_j$ with $\theta_0=1-\theta_1-\theta_2$, and $\psi=\phi_0^{\theta_0}\phi_1^{\theta_1}\phi_2^{\theta_2}$.
\end{Lemma}
\begin{proof}
A direct and routine adaptation of Lemmata 4.4 and 4.5 from \cite{Ase:01}, combined with similar arguments, yields the result.
\end{proof}
We now state a result that generalizes Theorem 6.2 of \cite{Ase:01} within our framework:
\begin{Thm}
Let $\gamma=\cdot^\theta$ and let $\phi_0,\phi_1\in\B$. Define $f=\phi_0/\phi_1$, $\psi=\phi_0/(\gamma\circ f)$, and
\[
 \psi\circ \omega=\frac{(\phi_0\circ \omega_0).(\gamma\circ \phi_1\circ \omega_1)}{\gamma\circ \phi_0\circ \omega_0}.
\]
Let $p,p_0,p_1,q\in[1,\infty]$ satisfy $1/p=(1-\theta)/p_0+\theta/p_1$. If $\lb(f)>0$ or $\ub(f)<0$ and if $0<\lb(\gamma)$, $\ub(\gamma)<1$, then
\[
 K^\gamma_q(\Lambda_{p_0}^{\phi_0}(\omega_0 d\mu),\Lambda_{p_1}^{\phi_1}(\omega_1 d\mu))
\]
coincides with the block-Lorentz-type space whose quasi-norm is given by
\[
 \|a\|  =
 \left(\sum_{k\in \Z}(\|a\chi_{\Omega^{\omega_0/\omega_1}_k}\|_{\Lambda_p^{\phi}(\omega d
 \mu)})^q\right)^{1/q}.
\]
\end{Thm}
\begin{proof}
This result is obtained by extending the proof of Theorem 6.2 in \cite{Ase:01}, using Lemma~\ref{lem:lor} to accommodate the generalized framework.
\end{proof}


\begin{thebibliography}{99}

\bibitem{Ase:97}
I.~Asekritova and N.~Krugljak.
\newblock On equivalence of $K$- and $J$-methods for $(n+1)$-tuples of {B}anach spaces.
\newblock {\em Stud. Math.}, 122:99--116, 1997.

\bibitem{Ase:01}
I.~Asekritova, N.~Kruglyak, L.~Maligranda, L.~Nikolova and L-E.~Persson.
\newblock Lions-{P}eetre reiteration formulas for triples and their applications.
\newblock {\em Stud. Math.}, 145:219--254, 2001.

\bibitem{Bergh:76}
J.~Bergh and J.~L{\"o}fstr{\"o}m.
\newblock {\em Interpolation {S}paces: {A}n {I}ntroduction}.
\newblock Grundlehren der mathematischen Wissenschaften. Springer-Verlag, 1976.

\bibitem{Bingham:87}
N.~H. Bingham, C.~M. Goldie, and J.~L. Teugels.
\newblock {\em Regular {V}ariation}.
\newblock Cambridge University Press, 1987.

\bibitem{Boyd:67}
D.~W. Boyd.
\newblock The {H}ilbert {T}ransform on {R}earrangement-{I}nvariant {S}paces.
\newblock {\em Canadian J. Math.}, 19, 1967.

\bibitem{Boyd:69}
D.~W. Boyd.
\newblock Indices of function spaces and their relationship to interpolation.
\newblock {\em Canadian J. Math.}, 21:1245--1254, 1969.

\bibitem{Brudnyi:81}
Yu.A.~Brudny{\v\i} and N.Ja.~Krugljak.
\newblock Real interpolation functors.
\newblock {\em Soviet. Math. Dokl.}, 23:5--8, 1981.

\bibitem{Bru:91}
Yu.A.~Brudnyi and N.Ya.~Krugljak.
\newblock {\em Interpolation {F}unctors and {I}nterpolation {S}paces}.
\newblock North-Holland Mathematical Library, 1991.

\bibitem{Cobos:88}
F.~Cobos and D.~L. Fern{\'a}ndez.
\newblock {H}ardy-{S}obolev spaces and {B}esov-spaces with a function
  parameter.
\newblock {\em Lecture Notes in Math.}, 1302:158--170, 1988.

\bibitem{Cwikel:81}
M.~Cwikel and J.~Peetre.
\newblock Abstract {$K$} and {$J$} spaces.
\newblock {\em J. Math. Pures Appl.}, 60:1--50, 1981.

\bibitem{Farkas:06b}
A.~M. Caetano and W.~Farkas.
\newblock Local growth envelopes of {B}esov spaces of generalized smoothness.
\newblock {\em Z. fur Anal. ihre Anwend.}, 25:265--298, 2006.

\bibitem{Farkas:06}
W.~Farkas and H.-G. Leopold.
\newblock Characterisations of function spaces of generalised smoothness.
\newblock {\em Ann. Mat. Pura Appl.}, 185:1--62, 2006.

\bibitem{Foi:61}
C.~Foias and J.L. Lyons.
\newblock Sur certains th{\'e}or{\`e}mes d'interpolation.
\newblock {\em Acta Sci. Math. (Szeged)}, 22:269--282, 1961.

\bibitem{Gustavsson:78}
J.~Gustavsson.
\newblock A function parameter in connection with interpolation of banach
  spaces.
\newblock {\em Math. Scand.}, 60:33--79, 1978.

\bibitem{Haroske:08}
D.~Haroske and S.~Moura.
\newblock Continuity envelopes and sharp embeddings inspaces of generalized
  smoothness.
\newblock {\em J. Funct. Anal.}, 254, 2008.

\bibitem{Holmstedt:70}
T.~Holmstedt.
\newblock Interpolation of quasi-normed spaces.
\newblock {\em Math. Scand.}, 26:177--199, 1970.

\bibitem{Janson:81}
S.~Janson.
\newblock Minimal and maximal methods of interpolation.
\newblock {\em J. Funct. Anal.}, 44:50--73, 1981.

\bibitem{Ker:66}
N.~Kerzman.
\newblock Sur certains ensembles convexes li{\'e}s {\`a} des espaces {$L_p$}.
\newblock {\em C. R. Acad. Sci. Paris}, 263:365--367, 1966.

\bibitem{Krein:82}
S.~G. Krein, Ju.~I. Petunin, and E.~M. Semenov.
\newblock {\em Interpolation of {L}inear {O}perators}, volume~54.
\newblock Amer. Math. Soc., translations of mathematical monographs edition,
  1982.

\bibitem{Kreit:13}
D.~Kreit and S.~Nicolay.
\newblock Characterizations of the elements of generalized
  {H}{\"o}lder-{Z}ygmund spaces by means of their representation.
\newblock {\em J. Approx. Theory}, 172:23--36, 2013.

\bibitem{Kreit:18}
D.~Kreit and S.~Nicolay.
\newblock Generalized pointwise {H}{\"o}lder spaces defined via admissible
  sequences.
\newblock {\em J. Funct. Spaces}, 2018.

\bibitem{Lamby:22}
T.~Lamby and S.~Nicolay.
\newblock Some remarks on the {B}oyd functions related to the admissible sequences.
\newblock {\em Z. fur Anal. ihre Anwend.}, 41:211-227, 2022.

\bibitem{Lamby:24}
T.~Lamby and S.~Nicolay.
\newblock Interpolation with a function parameter from the category point of view.
\newblock {\em J. Funct. Anal.}, 286:110249, 2024.

\bibitem{Lamby:25}
T.~Lamby and S.~Nicolay.
\newblock Continuous interpolation spaces with a function parameter : properties and applications.
\newblock {\em Rocky Mountain J. Math..}, 2025.

\bibitem{Lamby:th}
T. Lamby.
\emph{Generalized Interpolation Methods and Pointwise Regularity through Continued Fractions and Diophantine Approximations}.
\newblock Ph.D. thesis, University of Li{\`e}ge, 2025.

\bibitem{Loosveldt:18}
L.~Loosveldt and S.~Nicolay.
\newblock Some equivalent definitions of {B}esov spaces of generalized
  smoothness.
\newblock {\em Math. Nach.}, 292:2262--2282, 2019.

\bibitem{Lorentz:50}
G.~Lorentz.
\newblock Some new function spaces
\newblock {\em Ann. Math.}, 51:37--55, 1950.

\bibitem{Maligranda:86}
L.~Maligranda.
\newblock Interpolation between sum and intersection of {B}anach spaces.
\newblock {\em J. Approx. Theory}, 47:42--53, 1986.

\bibitem{Marcinkiewicz:39}
J.~Marcinkiewicz.
\newblock Sur l'interpolation d'op{\'e}rateurs.
\newblock {\em C. R. Ac. Sci. Paris}, 208:1272--1273, 1939.

\bibitem{Matu:60}
W.~Matuszewska and W.~Orlicz
\newblock On certain properties of $\phi$-functions.
\newblock {\em Acad. Polon. Sci. S´er. Sci. Math. Astronom. Phys.}, 8:439--443, 1960.

\bibitem{Merucci:84}
C.~Merucci.
\newblock Applications of interpolation with a function parameter to {L}orentz, {S}obolev and {B}esov spaces.
\newblock In M.~Cwikel and J.~Peetre, editors, {\em Interpolation Spaces and Allied Topics in Analysis}, pages 183--201. Springer, 1984.

\bibitem{Ovchinnikov:84}
V.I.~Ovchinnikov.
\newblock {\em The Method of Orbits in Interpolation Theory}.
\newblock Harwood Academic Publishers, 1984.

\bibitem{Peetre:68}
J.~Peetre.
\newblock A theory of interpolation of normed spaces.
\newblock {\em Notas Math. Brasilia}, 39:1--86, 1968.

\bibitem{Persson:86}
L.E.~Persson.
\newblock Interpolation with a parameter function.
\newblock {\em Math. Scand.}, 59:199--222, 1986.

\bibitem{Riesz:27}
M.~Riesz.
\newblock Sur les maxima des formes bilin{\'e}aires et sur les fonctionnelles lin{\'e}aires.
\newblock {\em Acta Math.}, 49:465--497, 1927.

\bibitem{Spa:74}
G.~Sparr.
\newblock Interpolation of {S}everal {B}anach {S}paces.
\newblock {\em Ann. di Mat. Pura ed Appl.}, 99:247-316, 1974.

\bibitem{Thorin:48}
G.O.~Thorin.
\newblock Convexity theorems generalizing those of {M}. {R}iesz and {H}adamard with some applications.
\newblock {\em Comm. Sem. Math. Univ. Lund}, 9:1--58, 1948.

\bibitem{Triebel:78}
H.~Triebel.
\newblock {\em Interpolation theory, function spaces, differential operators}.
\newblock North-Holland Mathematical Library, 1978.

\bibitem{Yos:70}
A.~Yoshikawa.
\newblock Sur la th{\'e}orie d'espaces d'interpolation - les espaces de moyenne de plusieurs espaces de {B}anach.
\newblock {\em J. Fac. Sci. Univ. Tokyo}, 16:407--468, 1970.

\end{thebibliography}
\end{document}